\documentclass[a4paper,reqno]{amsart}
\usepackage[utf8]{inputenc}
\usepackage{xr-hyper}
\usepackage[colorlinks=false]{hyperref}
\usepackage{pdfsync}
\usepackage{verbatim}
\usepackage{amsthm,amsfonts,amssymb,mathtools}
\usepackage{xcolor}
\usepackage{graphicx}
\usepackage{tikz-cd}
\usepackage{enumitem}
\usepackage{stmaryrd}
\usepackage{bbm}
\usepackage{relsize}

\newcommand{\C}{\mathcal{C}}
\newcommand{\D}{\mathcal{D}}
\newcommand{\E}{\mathcal{E}}
\newcommand{\F}{\mathcal{F}}
\newcommand{\G}{\mathcal{G}}

\newcommand{\I}{\mathcal{I}}

\renewcommand{\L}{\mathcal{L}}

\renewcommand{\P}{\mathcal{P}}

\renewcommand{\S}{\mathcal{S}}
\newcommand{\T}{\mathcal{T}}
\newcommand{\U}{\mathcal{U}}

\newcommand{\X}{\mathfrak{X}}
\newcommand{\Y}{\mathcal{Y}}

\newcommand{\FF}{\mathbb{F}}

\newcommand{\RR}{\mathbb{R}}

\newcommand{\Cat}{\operatorname{Cat}}

\newcommand{\coker}{\operatorname{coker}}
\newcommand{\conv}{\operatorname{conv}}

\newcommand{\End}{\operatorname{End}}

\newcommand{\Ext}{\operatorname{Ext}}
\newcommand{\Fin}{\operatorname{Fin}_*}

\newcommand{\Fun}{\operatorname{Fun}}

\newcommand{\Hom}{\operatorname{Hom}}

\newcommand{\im}{\operatorname{im}}

\newcommand{\op}[1]{#1^{\operatorname{op}}}

\newcommand{\Set}{\operatorname{Set}}
\newcommand{\Sh}{\operatorname{Sh}}

\newcommand{\Top}{\operatorname{Top}}

\newcommand{\colim}{\varinjlim}
\newcommand{\congr}{\equiv}
\newcommand{\del}{\partial}
\newcommand{\disj}{\amalg}

\renewcommand{\epsilon}{\varepsilon}
\newcommand{\from}{\leftarrow}
\renewcommand{\hat}{\widehat}

\renewcommand{\iff}{\Leftrightarrow}

\newcommand{\mono}{\rightarrowtail}
\newcommand{\onto}{\twoheadrightarrow}
\newcommand{\isom}{\cong}
\renewcommand{\lim}{\varprojlim}
\newcommand{\longfrom}{\longleftarrow}
\newcommand{\longiff}{\Longleftrightarrow}

\newcommand{\longinto}{\lhook\joinrel\longrightarrow}

\newcommand{\longto}{\longrightarrow}

\newcommand{\minus}{\smallsetminus}

\renewcommand{\phi}{\varphi}
\renewcommand{\tilde}{\widetilde}

\newcommand{\xto}[1]{\xrightarrow{#1}}

\newcommand{\isofrom}{\xleftarrow{\ \raisebox{-2pt}[0pt][0pt]{\ensuremath{\sim}}\ }}
\newcommand{\isoto}{\xrightarrow{\ \raisebox{-2pt}[0pt][0pt]{\ensuremath{\sim}}\ }}

\newcommand{\pair}[1]{\langle{#1}\rangle}
\newcommand{\pullback}{\mathlarger{\mathlarger{\mathlarger{\mathlarger{\cdot\!\!\lrcorner}}}}}

\newcommand{\pushout}{\mathlarger{\mathlarger{\mathlarger{\mathlarger{\raisebox{-2.5pt}{\ensuremath{\ulcorner}}\!\!\cdot}}}}}

\DeclareRobustCommand{\skipTOCentry}[5]{}

\newtheorem{thm}{Theorem}[section]

\newtheorem{lemma}[thm]{Lemma}
\newtheorem{prop}[thm]{Proposition}
\theoremstyle{definition}
\newtheorem{defn}[thm]{Definition}
\newtheorem{ex}[thm]{Example}
\newtheorem{rem}[thm]{Remark}

\numberwithin{equation}{section}

\newcommand{\bico}[1]{[ #1 ]}
\newcommand{\lco}[1]{\langle #1 ]}
\newcommand{\rco}[1]{[ #1 \rangle}
\newcommand{\lP}{P^{\triangleleft}}
\newcommand{\rP}{P^{\hspace{.5pt}\triangleright}}
\newcommand{\SC}{S_{\oplus}}

\begin{document}
%------------------------------------------------------

\title{Higher Segal structures in algebraic $K$-theory}
\author{Thomas Poguntke}
\address{Hausdorff Center for Mathematics \\ Endenicher Allee 62, 53115 Bonn, Germany}
\email{thomas.poguntke@hcm.uni-bonn.de}
%\urladdr{http://www.math.uni-bonn.de/people/poguntke/}

\keywords{Algebraic K-theory, higher Segal spaces, Waldhausen S-construction}
\subjclass[2010]{18E05, 18E10, 18G30, 19D10, 55U10}

\begin{abstract}
{We introduce higher dimensional analogues of simplicial constructions due to Segal and Waldhausen, respectively producing the direct sum and algebraic $K$-theory spectra of an exact category. We then investigate their fibrancy properties, based on the formalism of higher Segal spaces by Dyckerhoff-Kapranov.}
\end{abstract}

\date{\today}
\maketitle

{\hypersetup{linkbordercolor=white} \tableofcontents}

\section*{Introduction}\label{sec0}

Let $\E$ be an exact category. In this article, we investigate certain simplicial categories naturally arising in the context of algebraic $K$-theory. They are obtained as generalizations of the following ubiquitous constructions, which produce a simplicial category from $\E$;

\begin{itemize}
\item the Segal construction $\SC(\E)$,
\item the Waldhausen construction $S(\E)$.
\end{itemize}

From a topological perspective, their relevance lies in the fact that they provide deloopings of the direct sum and algebraic $K$-theory spectra of $\E$, respectively.

From an algebraic perspective, both constructions have fibrancy properties of structural importance. Namely, $\SC(\E)$ and $S(\E)$ satisfy the $1$-Segal and $2$-Segal conditions, respectively, modelling the structure of an associative, resp. multi-valued associative, monoid.

This work centers around certain higher dimensional generalizations of the above;

\begin{itemize}
\item the $k$-dimensional Segal construction $\SC^{\pair k}(\E)$,
\item the $k$-dimensional Waldhausen construction $S^{\pair k}(\E)$.
\end{itemize}

For $k=1$, we recover $\SC(\E)$ and $S(\E)$, respectively. For $k=2$, these constructions form the foundational basis of real algebraic $K$-theory as introduced by Hesselholt-Madsen \cite{HM}, and studied further for example by Dotto \cite{Dotto}. In our context, $\E$ is not endowed with a duality structure, and we consider these objects purely as simplicial categories.

Similarly to the case $k=1$, the higher dimensional constructions provide higher deloopings of algebraic $K$-theory and its split variant. Their relevance from the algebraic perspective warrants further investigation, for which our main results lay the groundwork.

\begin{thm}\label{mainthm}
The higher Segal construction $\SC^{\pair k}(\E)$ is lower $(2k-1)$-Segal. If $\E$ is an abelian category, then the higher Waldhausen construction $S^{\pair k}(\E)$ is $2k$-Segal.
\end{thm}

For the precise statements, see Theorem \ref{highersegal} and Theorem \ref{waldhausen2k}, respectively.

Higher Segal objects were introduced in \cite{HSS1}, with a focus on the $2$-Segal conditions, in particular showing that they are responsible for the associativity of Hall and Hecke algebras. From a different perspective, unital $2$-Segal spaces were defined and studied independently in \cite{GKT} and its series of sequels.

Further work in this area includes a precise description of unital $2$-Segal sets in terms of double categories in \cite{BOORS}, and the introduction of relative Segal conditions in \cite{Tashi} and \cite{Matt}, which model the structure of modules over multi-valued associative monoids.

Let us briefly outline the structure of the paper. In $\mathsection$\ref{sec1}, we summarize some basic theory of cyclic polytopes and their triangulations (from \cite{Ram} and \cite{Zieg}), which is used in $\mathsection$\ref{sec2} to define and study relations between the higher Segal conditions, as in \cite{HSS2}.

In $\mathsection$\ref{sec3}, we define the simplicial category $\SC^{\pair k}(\E)$ and prove the first part of Theorem \ref{mainthm}.

Section $\mathsection$\ref{sec4} introduces the homological context for the second part of Theorem \ref{mainthm}, which is finally proven in $\mathsection$\ref{sec5}, after providing the definition of $S^{\pair k}(\E)$ and several examples.

\addtocontents{toc}{\skipTOCentry}
\section*{Acknowledgements}
I would like to thank T. Dyckerhoff for suggesting this topic, his constant interest and invaluable contributions, as well as several enlightening discussions with G. Jasso. I would also like to thank Walker Stern for helpful suggestions.

\section{Cyclic polytopes}\label{sec1}

In this section, we recall results on polytopes relevant to the study of higher Segal objects.

\begin{defn}
Let $d \geq 0$, and consider the moment curve
\[\gamma_d\!: \RR \longto \RR^d,\ t \longmapsto (t,t^2,t^3,\ldots,t^d).\]
For a finite subset $N \subseteq \RR$, the $d$-dimensional cyclic polytope on the vertices $N$ is defined to be the convex hull of the set $\gamma_d(N) \subseteq \RR^d$, and denoted by
\[C(N,d) = \operatorname{conv}(\gamma_d(N)).\]
\end{defn}

The combinatorial type of the polytope $C(N,d)$ only depends on the cardinality of $N$. We will usually consider $N$ to be the set $[n] = \{0,1,\dots,n\}$, where $n \geq 0$.

Cyclic polytopes are simplicial polytopes, i.e., their boundary forms a simplicial complex, which organizes into two components (with non-empty intersection), as follows.

\begin{defn}
A point $x$ in the boundary of $C([n],d+1)$ is called a lower point, if
\[(x + \RR_{>0}) \cap C([n],d+1) = \emptyset,\]
where the half-line of positive real numbers $\RR_{>0} \subseteq \RR^{d+1}$ is embedded into the last coordinate. Similarly, $x$ is said to be an upper point, if
\[(x - \RR_{>0}) \cap C([n],d+1) = \emptyset.\]
\end{defn}

The lower and upper points of the boundary form simplicial subcomplexes of $C([n],d+1)$, which admit the following purely combinatorial description.

\begin{defn}
Let $I \subseteq [n]$. A gap of $I$ is a vertex $j \in [n] \minus I$ in the complement of $I$. A gap is said to be even, resp. odd, if the cardinality $\#\{i \in I \mid i > j\}$ is even, resp. odd. The subset $I$ is called even, resp. odd, if all gaps of $I$ are even, resp. odd.
\end{defn}

\begin{prop}[Gale's evenness criterion; \cite{Zieg}, Theorem 0.7]\label{gale}
Let $n \geq 0$, and let $I \subseteq [n]$ with $\#I = d+1$. Then $I$ defines a $d$-simplex in the lower, resp. upper, boundary of $C([n],d+1)$ if and only if $I$ is even, resp. odd.
\end{prop}

Forgetting the last coordinate of $\RR^{d+1}$ defines a projection map
\[p\!: C([n],d+1) \longto C([n],d).\]
For any $I \subseteq [n]$ with $\#I - 1 = r \leq d$, the projection $p$ maps the geometric $r$-simplex
\[|\Delta^I| \subseteq C([n],d+1)\]
homeomorphically to an $r$-simplex in $C([n],d)$.

\begin{defn}
The lower triangulation $\T_{\ell}$ of the polytope $C([n],d)$ is the triangulation given by the projections under $p$ of the simplices contained in the lower boundary of $C([n],d+1)$. Similarly, the upper triangulation $\T_u$ of $C([n],d)$ is defined by the projections of the simplices contained in the upper boundary of the polytope $C([n],d+1)$.
\end{defn}

Vice versa, any triangulation of $C([n],d)$ induces a piecewise linear section of $p$, whose image defines a simplicial subcomplex of $C([n],d+1)$. This interplay between the cyclic polytopes in different dimensions is what makes their combinatorics comparatively tractible.

\begin{defn}
Given a set $\I \subseteq 2^{[n]}$ of subsets of $[n]$, we denote by
\begin{equation}\label{subcomplex}
\Delta^{\I} \subseteq \Delta^n
\end{equation}
the simplicial subset of $\Delta^n$ whose $m$-simplices are given by those maps $[m] \to [n]$ which factor over some $I \in \I$.
\end{defn}

From the above discussion, it follows that we have canonical homeomorphisms
\[|\Delta^{\T_{\ell}}| \isom C([n],d), \mbox{ and } |\Delta^{\T_u}| \isom C([n],d),\]
expressing the lower, resp. upper, triangulation of $C([n],d)$ geometrically.

\begin{defn}
Let $I,J \subseteq [n]$ be subsets of cardinality $d+1$, as well as $|\Delta^I|$ and $|\Delta^J|$ the geometric $d$-simplices in $C([n],d) \subseteq \RR^d$ they define, respectively. Let $L(|\Delta^J|)$ denote the set of lower boundary points of $|\Delta^J| = C(J,d)$, and similarly, let $U(|\Delta^I|)$ be the set of upper boundary points of $|\Delta^I| = C(I,d)$. We write
\[|\Delta^I| \prec |\Delta^J| \longiff |\Delta^I| \cap |\Delta^J| \subseteq U(|\Delta^I|) \cap L(|\Delta^J|).\]
If $|\Delta^I| \prec |\Delta^J|$, then we say that $|\Delta^I|$ lies below the simplex $|\Delta^J|$.
\end{defn}

\begin{prop}[\cite{Ram}, Corollary 5.9]\label{transitive}
The transitive closure of $\prec$ defines a partial order on the set of nondegenerate $d$-simplices in $\Delta^n$.
\end{prop}

\begin{rem}\label{maxsimp}
Suppose that $\T \subseteq 2^{[n]}$ consists of subsets of $[n]$ of cardinality $d+1$ and defines a triangulation of the cyclic polytope $C([n],d)$. In particular,
\[|\Delta^{\T}| \isom C([n],d).\]
Let $I_0 \in \T$. Then either $L(|\Delta^{I_0}|)$ is contained in $L(|\Delta^{\T}|)$, or there is some $I_1 \in \T$ such that the simplex $|\Delta^{I_0}|$ lies below $|\Delta^{I_1}|$. Iterating this argument, we obtain a chain
\[|\Delta^{I_0}| \prec |\Delta^{I_1}| \prec |\Delta^{I_2}| \prec \ldots\]
of subsimplices of $|\Delta^{\T}|$. The statement of Proposition \ref{transitive} implies that this chain is acyclic and therefore has to terminate after finitely many steps. Thus, there exists $I_{\infty} \in \T$ with
\[L(|\Delta^{I_{\infty}}|) \subseteq L(|\Delta^{\T}|).\]
Similarly, there exists a set $I_{-\infty} \in \T$ such that $U(|\Delta^{I_{-\infty}}|) \subseteq U(|\Delta^{\T}|)$.
\end{rem}

\section{Higher Segal conditions}\label{sec2}

Let $\C$ be an $\infty$-category which admits finite limits. Following \cite{HSS2}, we introduce a framework of fibrancy properties of simplicial objects in $\C$ governed by the combinatorics from $\mathsection$\ref{sec1} of cyclic polytopes and their triangulations.

\begin{defn}
For $n \geq d \geq 0$, we introduce the lower subposet of $2^{[n]}$ as follows;
\[\L([n],d) = \{J \mid J \subseteq I \mbox{ for some even } I \subseteq [n] \mbox{ with } \#I = d+1\}.\]
Analogously, we define
\[\U([n],d) \subseteq 2^{[n]}\]
as the poset of all subsets contained in an odd subset $I \subseteq [n]$ of cardinality $\#I = d+1$.
\end{defn}

By Proposition \ref{gale}, the sets of subsimplices of $C([n],d)$ described by $\L([n],d)$ and $\U([n],d)$ define the lower and upper triangulations of the cyclic polytope, respectively.

\begin{defn}
Let $d \geq 0$. A simplicial object $X \in \C_{\Delta}$ is called
\begin{itemize}
\item lower $d$-Segal if, for every $n \geq d$, the natural map
\[X_n \longto \lim_{I \in \L([n],d)} X_I\]
is an equivalence;
\item upper $d$-Segal if, for every $n \geq d$, the natural map
\[X_n \longto \lim_{I \in \U([n],d)} X_I\]
is an equivalence;
\item fully $d$-Segal if $X$ is both lower and upper $d$-Segal.
\end{itemize}
\end{defn}

\begin{ex}\label{Segalex}
Let $X \in \C_{\Delta}$ be a simplicial object.
\begin{enumerate}[label=$(\arabic*)$]
\item Then $X$ is lower (or upper) $0$-Segal if and only if $X \isom X_0$ is equivalent to the constant simplicial object on its $0$-cells.
\item\label{1Segal} The simplicial object $X$ is lower $1$-Segal if, for every $n \geq 1$, the map
\[X_n \longto X_{\{0,1\}} \times_{X_{\{1\}}} X_{\{1,2\}} \times_{X_{\{2\}}} \cdots \times_{X_{\{n-1\}}} X_{\{n-1,n\}}\]
is an equivalence. That is, $X$ is a Segal object in the sense of Rezk \cite{Rezk}.

For $X \in \Set_{\Delta}$, this means that $X$ is equivalent to the nerve of the category with objects $X_0$, morphisms $X_1$, and composition induced by the correspondence
\[X_{\{0,1\}} \times_{X_{\{1\}}} X_{\{1,2\}} \isofrom X_2 \longto X_{\{0,2\}}.\]
Furthermore, $X \in \Set_{\Delta}$ is fully $1$-Segal if and only if this defines a discrete groupoid. In fact, in general, an object $X \in \C_{\Delta}$ is upper $1$-Segal if, for every $n \geq 1$, we have
\[X_n \isoto X_{\{0,n\}}.\]
\item The simplicial object $X$ is lower $2$-Segal if, for every $n \geq 2$, the map
\[X_n \longto X_{\{0,n-1,n\}} \times_{X_{\{0,n-1\}}} X_{\{0,n-2,n-1\}} \times_{X_{\{0,n-2\}}} \cdots \times_{X_{\{0,2\}}} X_{\{0,1,2\}}\]
is an equivalence. Similarly, $X$ is upper $2$-Segal if, for every $n \geq 2$,
\[X_n \isoto X_{\{0,1,n\}} \times_{X_{\{1,n\}}} X_{\{1,2,n\}} \times_{X_{\{2,n\}}} \cdots \times_{X_{\{n-2,n\}}} X_{\{n-2,n-1,n\}}.\]
It follows that $X$ is fully $2$-Segal if and only if it is $2$-Segal in the sense of \cite{HSS1}. This is most readily seen by reducing to (1-)Segal objects as in \ref{1Segal} by applying the respective path space criteria, Proposition \ref{pathcrit} and \cite{HSS1}, Theorem 6.3.2.
\end{enumerate}
\end{ex}

\begin{rem}
Let $X$ be a simplicial object in an $\infty$-category $\C$ which admits limits. Then we can form the right Kan extension of $X$ along the (opposite of the) Yoneda embedding
\[\op\Y\!: N(\op\Delta) \longto \op{\Fun(N(\op\Delta), \Top)},\]
where $\Top$ denotes the $\infty$-category of spaces. In particular, by means of this extension, we may evaluate $X$ on any simplicial set. Then we can reformulate the higher Segal conditions as follows. The simplicial object $X$ is lower $d$-Segal if and only if, for every $n \geq d$, the inclusion
\[\Delta^{\L([n],d)} \subseteq \Delta^n\]
is an $X$-equivalence, that is, maps to an equivalence in $\C$ under $X$ (cf. \cite{HSS1}, Proposition 5.1.4). Similarly, $X$ is upper $d$-Segal if and only if, for every $n \geq d$, the inclusion
\[\Delta^{\U([n],d)} \subseteq \Delta^n\]
is an $X$-equivalence.
\end{rem}

\begin{prop}\label{triang}
Let $d \geq 0$, and let $X$ be a simplicial object in an $\infty$-category $\C$ with limits. Then $X$ is fully $d$-Segal if and only if, for every $n \geq d$, and every triangulation of the cyclic polytope $C([n],d)$ defined by the poset of simplices $\T \subseteq 2^{[n]}$, the natural map
\[X_n \longto \lim_{I \in \T} X_I\]
is an equivalence.
\end{prop}
\begin{proof}
By \cite{Ram}, Corollary 5.12, any triangulation $\T$ of $C([n],d)$ can be connected to the lower and upper triangulations $\T_{\ell}$ and $\T_u$ via a sequence of elementary flips of the form
\[L(|\Delta^I|) \subseteq |\Delta^I| \supseteq U(|\Delta^I|).\]
This implies that we have a zig-zag of $X$-equivalences connecting $\Delta^{\T}$ to $\Delta^{\T_{\ell}}$ and $\Delta^{\T_u}$ in the category of simplicial sets over $\Delta^n$. By $2$-out-of-$3$, this implies that the inclusion $\Delta^{\T} \subseteq \Delta^n$ is again an $X$-equivalence, which is what was to be shown.

For the converse, there is nothing to prove.
\end{proof}

\begin{defn}
Let $X$ be a simplicial object in an $\infty$-category $\C$. The left path space $\lP X$ is the simplicial object in $\C$ defined as the pullback of $X$ along the endofunctor
\[\Delta \to \Delta,\ [n] \mapsto [0] \oplus [n].\]
Here, for two linearly ordered sets $I$ and $J$, the ordinal sum $I \oplus J$ is the disjoint union $I \disj J$ of sets, endowed with the linear order where $i \leq j$ for every pair of $i \in I$ and $j \in J$.

Similarly, the right path space $\rP X$ is given by the pullback of $X$ along
\[\Delta \to \Delta,\ [n] \mapsto [n] \oplus [0].\]
\end{defn}

In terms of these constructions, we have the following criteria for various Segal conditions.

\begin{prop}[Path space criterion]\label{pathcrit}
Let $d \geq 0$, and let $\C$ be an $\infty$-category with finite limits. Let $X \in \C_{\Delta}$ be a simplicial object.
\begin{itemize}
\item Suppose $d$ is even. Then
\begin{enumerate}[label=$(\arabic*)$]
\item $X$ is lower $d$-Segal if and only if $\lP X$ is lower $(d-1)$-Segal,
\item $X$ is upper $d$-Segal if and only if $\rP X$ is lower $(d-1)$-Segal.
\end{enumerate}
\item Suppose $d$ is odd. Then the following conditions are equivalent.
\begin{enumerate}[label=$(\roman*)$]
\item $X$ is upper $d$-Segal.
\item $\lP X$ is upper $(d-1)$-Segal.
\item $\rP X$ is lower $(d-1)$-Segal.
\end{enumerate}
\end{itemize}
\end{prop}
\begin{proof}
We show that if $d$ is even, then $X$ is an upper $d$-Segal object if and only if $\rP X$ is lower $(d-1)$-Segal. All other assertions follow by analogous arguments.

However, the claim is an immediate consequence of the observation that the map $I \mapsto I \oplus [0]$ defines a bijection between the maximal elements of the posets $\L([n-1],d-1)$ and $\U([n],d)$. Namely, a subset $I \subseteq [n]$ of cardinality $d+1$ is odd if and only if $n \in I$ and $I \minus \{n\}$ is an even subset of $[n-1]$.
\end{proof}

\begin{ex}
Applying the bijections from the above proof twice to Example \ref{Segalex} \ref{1Segal} says that the maximal elements of $\U([n],3)$ are precisely given by
\[\{\{0,1,2,n\},\{0,2,3,n\},\ldots,\{0,n-2,n-1,n\}\}.\]
\end{ex}

\begin{rem}
Note that there is no path space criterion for lower $d$-Segal objects when $d$ is odd. For example, $\L([5],3)$ has maximal elements
\[\{\{0,1,2,3\},\{0,1,3,4\},\{0,1,4,5\},\{1,2,4,5\},\{2,3,4,5\},\{1,2,3,4\}\},\]
whereas $I \mapsto [0] \oplus I$, resp. $I \oplus [0]$, maps the maximal elements of $\L([4],2)$, resp. $\U([4],2)$, to
\[\{\{0,1,2,3\},\{0,1,3,4\},\{0,1,4,5\}\}, \mbox{ resp. } \{\{0,1,4,5\},\{1,2,4,5\},\{2,3,4,5\}\}.\]

That is, these maps are not even jointly surjective, and their images intersect. Rather, the complement of the image of $I \mapsto I \oplus [0]$ is given by the maximal elements of $\L([n-1],d)$. It follows by induction that we have a decomposition of sets (of objects)
\[\L([n],d) = \bigcup_{i=d}^n \U([i-1],d-1) \oplus \{i\},\]
which is disjoint on maximal elements.
\end{rem}

\begin{prop}\label{implications}
Let $X$ be a simplicial object in an $\infty$-category $\C$ which admits limits. Assume that $X$ is lower or upper $d$-Segal. Then $X$ is fully $(d+1)$-Segal.
\end{prop}
\begin{proof}
We show the statement assuming that $X$ is lower $d$-Segal. The proof for upper	$d$-Segal objects is similar. Let $n \geq d+1$ and consider a collection $\T$ defining a triangulation $|\Delta^{\T}|$ of the cyclic polytope $C([n],d+1)$. Recall that $\T_{\ell}$ defines the simplicial complex $L(|\Delta^{\T}|)$ of lower facets, and the projection $p\!: C([n],d+1) \to C([n],d)$ identifies $|\Delta^{\T_{\ell}}| \subseteq C([n],d+1)$ with the simplicial subcomplex defining the lower triangulation
\[p(|\Delta^{\T_{\ell}}|) \subseteq C([n],d).\]
Thus, we obtain a commutative diagram
\begin{equation}\label{dsegal}
  \begin{tikzcd}
    \Delta^{\T} \ar[r, "{\iota}"] & \Delta^n \\
    \Delta^{\T_{\ell}} \ar[u, "{\kappa}"] \ar[ur, "{\iota_{\ell}}"'] & 
  \end{tikzcd}
\end{equation}
of simplicial sets, in which by definition, $\iota \in \S_{d+1}$ and $\iota_{\ell} \in \S_{d}$. That is, $\S_{d+1}$ denotes the set of $(d+1)$-Segal coverings (cf. \cite{HSS1}, $\mathsection$5.2), while $\S_d$ denotes the set of lower $d$-Segal coverings, so that $\S_d$-local objects are precisely the lower $d$-Segal objects.

In order to deduce $(d+1)$-Segal descent for $X$, we have to show that $\iota \in \overline{\S_d}$, by which we denote the set of $\S_d$-equivalences, by the tautological part of Proposition \ref{triang}.

By the $2$-out-of-$3$ property of $\overline{\S_d}$, it suffices to show that $\kappa \in \overline{\S_d}$. We will do so by showing that $\kappa$ can be obtained as an iterated pushout along morphisms in $\S_d$.

By Remark \ref{maxsimp}, the triangulation $|\Delta^{\T}|$ of $C([n],d+1)$ contains a maximal $(d+1)$-simplex of the form $|\Delta^{I}|$, defined by a singleton collection $\{I\} \subseteq \binom{[n]}{d+2}$. Let $\I_{\ell}$ be the set which defines the lower facets of $|\Delta^{I}|$, defining a triangulation $|\Delta^{\I_{\ell}}|$ of $C(I, d)$. Then the inclusion of simplicial sets
\[\kappa_{\ell}\!: \Delta^{\I_{\ell}} \longto \Delta^I\]
is contained in $\S_d$. Further, since $|\Delta^{I}|$ is a maximal simplex, we have a pushout diagram 
\[\begin{tikzcd}
  \Delta^{\T} \ar[dr, phantom, "{\pullback}" very near start] & \ar[l] \Delta^I \\
  \Delta^{\T^{(1)}} \ar[u, "{\kappa^{(1)}}"] & \ar[l] \ar[u, "{\kappa_{\ell}}"'] \Delta^{\I_{\ell}}
\end{tikzcd}\]
of simplicial sets, where $\T^{(1)} = \T \minus \{I\}$. Thus, the map $\kappa^{(1)}$ lies in $\overline{\S_d}$, and $|\Delta^{\T^{(1)}}|$ is an admissible simplicial subcomplex of $C([n],d+1)$ with one $(d+1)$-simplex less than $|\Delta^{\T}|$.

Assume that the triangulation $|\Delta^{\T}|$ consists of exactly $r$ simplices of dimension $(d+1)$. By iterating the argument just given, we obtain a chain of morphisms
\[\Delta^{\T_{\ell}} \longinto \Delta^{\T^{(r-1)}} \longinto \dots \longinto \Delta^{\T^{(1)}} \longinto \Delta^{\T}\]
in $\overline{\S_d}$ whose composite is the morphism $\kappa$ from \eqref{dsegal}. Thus, it is also contained in $\overline{\S_d}$.
\end{proof}

\section{The higher Segal construction}\label{sec3}

Let $\D$ be a pointed category with finite products. In this section, we study a generalization of a construction due to Segal \cite{Segal} which is similar to a construction proposed (for $k=2$) by Hesselholt and Madsen \cite{HM}. The higher dimensional variants (for $k \geq 3$) are straightforward to define, but do not seem to have appeared in the literature as of yet.

Let $\Fin$ denote the category of finite pointed sets. For $T\in\Fin$, we denote by $\P(T)$ its poset of pointed subsets, considered as a small category.

\begin{defn}
Let $T \in \Fin$ be a finite pointed set. A $\D$-valued presheaf
\[\F\!: \op{\P(T)} \longto \D\]
on $\P(T)$ is called a sheaf if, for every pointed subset $U \subseteq T$, the canonical map
\[\F(U) \longto \prod_{u \in U \minus \{*\}} \F(\{*,u\})\]
is an isomorphism. We denote by $\Sh(T,\D)$ the category of $\D$-valued sheaves on $\P(T)$.
\end{defn}

Given a map $\rho\!: T \to T'$ in $\Fin$, we define the pointed preimage functor
\[\rho^{\times}\!: \P(T') \longto \P(T),\ U \longmapsto \rho^{-1}(U \minus \{*\}) \disj \{*\}.\]
Then the direct image functor $\F \mapsto \rho_*\F = \F \circ \rho^{\times}$ makes the assignment
\[\Sh(-,\D)\!: \Fin \longto \Cat\]
into a functor with values in the category of small categories.

\begin{defn}
Let $k \geq 1$. The $k$-dimensional Segal construction of $\D$ is defined to be the simplicial category
\[\SC^{\pair k}(\D) = \Sh(S^k, \D) \in \Cat_{\Delta},\]
where $S^k = \Delta^k/\partial \Delta^k$ is considered as a simplicial object in $\Fin$.
\end{defn}

The goal of this section is to prove the following result, which is due to Segal \cite{Segal} in the case $k=1$. Throughout, a lower, resp. upper, $d$-Segal category means a lower, resp. upper, $d$-Segal object in $\Cat$, which is not to be confused with a Segal category in the sense of \cite{DKS}.

\begin{thm}\label{highersegal}
Let $k \geq 1$, and assume that $\D$ is a finitely cocomplete pointed category with finite products. Then the $k$-dimensional Segal construction $\SC^{\pair k}(\D)$ is a lower $(2k-1)$-Segal category. In particular, $\SC^{\pair k}(\D)$ is fully $2k$-Segal.
\end{thm}
\begin{proof}
The last part is an application of Proposition \ref{implications}. Now let $n \geq 2k-1$, and set
\[\L = \L([n], 2k-1).\]
For $I,J \in \L$ with $I \supseteq J$, we denote by $\rho_{I,J}\!: S^k_I \to S^k_J$ the corresponding map of pointed sets, and further write $\rho_I = \rho_{[n],I}$ for brevity. We have to show that the canonical functor
\begin{equation}\label{segalsegal}
\Sh(S^k_n,\D) \longto \lim_{I \in \L} \Sh(S^k_I, \D)
\end{equation}
is an equivalence of categories. Consider the functor
\[\P\!: \L \longto \Cat,\ I \longmapsto \P(S^k_I).\]
We form the following version of its Grothendieck construction
\[\pi\!: \X_{\P} \longto \op\L.\]
The category $\X_{\P}$ has objects $(I,U)$, where $I \in \L$ and $U \in \P(S^k_I)$, and there is a unique morphism $(I,U) \leq (J,V)$ if $I \supseteq J$ and $U \subseteq \rho_{I,J}^{\times} V$. The functor $\pi$ is a cartesian fibration, where a morphism $(I,U) \leq (J,V)$ is cartesian if
\[U = \rho_{I,J}^{\times} V.\]
The category $\lim_{I \in \L} \Sh(S^k_I, \D)$ can be identified with the full subcategory of $\Fun(\op{\X_{\P}}, \D)$ spanned by those presheaves $\F$ which satisfy the following conditions.
\begin{enumerate}[label=$(\alph*)$]
\item\label{lim1} The presheaf $\F$ maps cartesian morphisms to isomorphisms in $\D$.
\item\label{lim2} For every $I \in \L$, the restriction of $\F$ to the fibre $\pi^{-1}(I) = \P(S^k_I)$ is a sheaf.
\end{enumerate}
A $\D$-valued presheaf on $\P(S^k_n)$ defines a presheaf on $\X_{\P}$ via pullback along the functor
\[\phi_0\!: \X_{\P} \longto \P(S^k_n),\ (I,U) \longmapsto \rho_I^{\times}U.\]
The lower Segal functor \eqref{segalsegal} is then obtained by restricting this pullback functor along $\phi_0$ to the category of sheaves on $\P(S^k_n)$.

Since the functor $\phi_0$ maps cartesian morphisms to the identity map in $\P(S^k_n)$, it factors over a unique functor $\phi\!: L\X_{\P} \to \P(S^k_n)$, where $L\X_{\P}$ denotes the localization of $\X_{\P}$ along the set of cartesian morphisms. Note further that imposing condition \ref{lim1} on a presheaf $\F$ on the category $\X_{\P}$ is equivalent to the requirement that $\F$ factors through $L\X_{\P}$.

We obtain an adjunction of presheaf categories as follows,
\begin{equation}\label{adj1}
\phi_!\!: \D_{L\X_{\P}} \longleftrightarrow \D_{\P(S^k_n)} :\!\phi^*
\end{equation}
where the functor $\phi^*$ maps the subcategory of sheaves to the subcategory $\lim_{I \in \L} \Sh(S^k_I, \D)$. Finally, we introduce the sheafification functor
\[\sigma\!: \D_{\P(S^k_n)} \longto \Sh(S^k_n,\D)\]
as the left adjoint to the inclusion. Then \eqref{adj1} induces an adjunction,
\begin{equation}\label{adj2}
\sigma \circ \phi_!\!: \lim_{I \in \L} \Sh(S^k_I, \D) \longleftrightarrow \Sh(S^k_n,\D) :\!\phi^*
\end{equation}
which we claim to be a pair of mutually inverse functors. In order to verify this, we show that the unit and counit are isomorphisms. For the former, it suffices to show that, for every sheaf $\G \in \Sh(S^k_n,\D)$ and every subset $\{*,\alpha\} \subseteq S^k_n$ of cardinality $2$, the unit morphism
\begin{equation}\label{unit1}
(\phi_! \phi^* \G)(\{*,\alpha\}) \longto \G(\{*,\alpha\})
\end{equation}
is an isomorphism. We have
\[(\phi_! \phi^* \G)(\{*,\alpha\}) \isom \colim_{\substack{(I,U) \in \op{L\X_{\P}} \\
\alpha \in \rho_I^{\times}(U)}} \G(\rho_I^{\times}(U)).\]
According to Lemma \ref{bars} \ref{bars1} below, the indexing category $\op \phi/\{*,\alpha\}$ of the colimit has a final object $(I_{\alpha},\{*,\alpha|_{I_\alpha}\})$, with $\rho_{I_{\alpha}}^{\times}(\{*,\alpha|_{I_\alpha}\}) = \{*,\alpha\}$. This immediately implies that \eqref{unit1} is indeed an isomorphism.

We proceed to prove that, for every object $\F \in \lim_{I \in \L} \Sh(S^k_I, \D)$, the counit morphism
\[\F \longto \phi^* \sigma \phi_! \F\]
is invertible. Similarly as above, it suffices to show for all $(J,\{*,\beta\}) \in L\X_{\P}$ that the map
\begin{equation}\label{counit1}
\F(J, \{*,\beta\}) \longto (\phi^* \sigma \phi_! \F)(J,\{*,\beta\})
\end{equation}
is an isomorphism in $\D$. Using Lemma \ref{bars} \ref{bars1}, we compute the right-hand side as
\[(\sigma \phi_! \F)(\rho_J^{\times}\{*,\beta\}) \isom \prod_{\alpha \in \rho_J^{-1}(\beta)} (\phi_! \F)(\{*,\alpha\}) \isom \prod_{\alpha \in \rho_J^{-1}(\beta)} \F(I_{\alpha}, \{*,\alpha|_{I_{\alpha}}\}).\]
Then Lemma \ref{bars} \ref{bars2} implies in particular that the map \eqref{counit1} is indeed an isomorphism.
\end{proof}

\begin{lemma}\label{bars}
In the terminology introduced in the proof of Theorem \ref{highersegal}, let $(J,V)$ be an object of the category $L\X_{\P}$. Then the following statements hold.
\begin{enumerate}[label=$(\arabic*)$]
\item\label{bars1} Let $\alpha \in \rho_J^{\times}(V) \minus \{*\}$. There is a unique morphism
\[(I_{\alpha},\{*,\alpha|_{I_{\alpha}}\}) \longto (J,V)\]
in $L\X_{\P}$, where
\[I_{\alpha} = \bigcup_{\alpha_i < \alpha_{i+1}} \{i,i+1\}.\]
\item\label{bars2} Let $\F$ be an object of $\lim_{I \in \L} \Sh(S^k_I, \D) \subseteq \D_{L\X_{\P}}$. There is an isomorphism
\[\F(J,V) \isoto \prod_{\alpha \in \rho_J^{\times}(V) \minus \{*\}} \F(I_\alpha,\{*,\alpha|_{I_{\alpha}}\}),\]
whose components are given by restriction along the unique morphisms from \ref{bars1}.
\end{enumerate}
\end{lemma}
\begin{proof}
The even subsets of $[n]$ of cardinality $2k$ are precisely the disjoint unions of $k$ subsets of the form $\{i,i+1\}$. Since $\alpha \not\congr *$, it follows that $I_{\alpha}$ is a (possibly non-disjoint) union of $k$ such subsets. However, it is contained in the even subset of $[n]$ of cardinality $2k$ obtained by inductively filling for each $\alpha_{i-1} < \alpha_{i} < \alpha_{i+1}$ either the maximal gap $j < i$ of $I_{\alpha}$ or its minimal gap $j > i$.

The key observation is that the subsets $I \in \L$ which contain $I_{\alpha}$ are exactly those with
\[\rho_I^{\times}(\rho_I(\{*,\alpha\})) = \{*,\alpha\}.\]
This implies that for a morphism in $L\X_{\P}$ of the form $(I_{\alpha},\{*,\alpha|_{I_{\alpha}}\}) \from (I,U) \to (J',V')$, we always have $U = \{*,\alpha|_{I}\}$. Thus, the only condition on $V'$ is that $\alpha|_{J'} \in V'$, and we can assume without loss of generality that $V = \{*,\alpha|_J\}$. In order to describe morphisms
\[\mu\!: (I_{\alpha},\{*,\alpha|_{I_{\alpha}}\}) \longto (J,\{*,\alpha|_J\})\]
in $L\X_{\P}$, we consider $\alpha$ as a sequence of $k$ bars situated in a diagram of $[n]$, signifying the fact that $\alpha_j < \alpha_{j+1}$ by the bar $j|(j+1)$. An object $(I,\{*,\alpha|_I\})$ corresponds to marking the elements $i \in I \subseteq [n]$, and the zig-zag $\mu$ is a sequence of moves which shift the markings. Each move consists of adding and then removing certain markings (adhering to the constraints imposed by the definition of $L\X_{\P}$).

The object $(I_{\alpha},\{*,\alpha|_{I_{\alpha}}\})$ marks all elements adjacent to a bar; that is, we visualize it as a diagram of the following exemplary form,
\[-- \bullet | \bullet - \bullet | \bullet | \bullet -- \bullet | \ \ldots \ | \bullet ---\]
where '$\bullet$' indicates a marked element and '$-$' an unmarked element of $[n]$. A single '$\bullet$' at a vertex $i \in [n]$ between two bars (i.e., '$| \bullet |$') corresponds precisely to the case $\alpha_{i-1} < \alpha_{i} < \alpha_{i+1}$ from above. Since $\alpha|_J \not\congr *$, this implies that $i \in J$; in fact, this condition states exactly that there is an element of $J$ between every pair of bars.

In order to see that $\mu$ is unique (if it exists), we first note that a '$\bullet$' can never cross a bar. Indeed, this would require a move
\[(I,\{*,\alpha|_I\}) \longfrom (H,\{*,\alpha|_H\}) \longto (\ldots)\]
which adds a marking to some $i \in I_{\alpha}$. Then we can define $\beta \in \P(S^k_H)$ by $\beta|_{H \minus \{i\}} \congr \alpha|_{H \minus \{i\}}$ and by replacing the jump $\alpha_{i-1} = \alpha_i < \alpha_{i+1}$ with $\beta_{i-1} < \beta_i = \beta_{i+1}$; but this contradicts the requirement that the left leg of the move be cartesian.

Then uniqueness follows from the fact that moves which are constrained within different sets of bars commute with one another, while the moves occuring between two particular bars all compose to the same shift of markings.

For the existence of $\mu$, we observe that after adding markings for each '$| \bullet |$' as described above (filling the gaps of $I_{\alpha}$; where we can always choose the gap closest to an element of $J$), we can remove at least one marking adjacent to each bar (with the exception of the '$| \bullet |$', in which case the vertex lies in $J$ already, as we have seen).

Then we can move each '$\bullet$' towards its intended position in $J$ by repeatedly marking the adjacent vertex and removing the original; moreover, once a '$\bullet$' has reached its destination, we can duplicate it. This requires no further sets of the form $\{i,i+1\}$ to cover all markings, that is, we stay within $\L$ in this process (as of course $J \in \L$ itself).

Finally, statement \ref{bars2} follows from the above, since $\F(J,V) = \F(\rho_J^{\times}V)$, and similarly,
\[\F(I_\alpha,\{*,\alpha|_{I_{\alpha}}\}) = \F(\{*,\alpha\});\]
but condition \ref{lim2} tells us that the restriction $\F|_{\P(S^k_J)}$ to the fibre $\pi^{-1}(J)$ is a sheaf.
\end{proof}

%\begin{rem}
%The objects $(I,\{*,\alpha|_I\}) \in \X_{\P}$ correspond bijectively to certain convex regions in the relevant cyclic polytope.
%\end{rem}

\section{Stringent categories}\label{sec4}

Let $\E$ be a proto-exact category (\cite{HSS1}, Definition 2.4.2; for example, any exact category). In this self-contained section, we carry out some homological algebraic preparations for the proof of the main result in $\mathsection$\ref{sec5}, which in particular requires no additive structure on $\E$.

\begin{defn}
A morphism $A \to B$ in $\E$ is called admissible if it factors as the composition
\[\begin{tikzcd}
  A \ar[rr] \ar[dr, two heads] & & B \\
   & C \ar[ur, tail] & 
\end{tikzcd}\]
of an admissible epimorphism and an admissible monomorphism in $\E$.

Consider a sequence of admissible morphisms together with their corresponding (unique up to unique isomorphism) factorizations as above,
\[\begin{tikzcd}
  A_{k} \ar[rr] \ar[dr, two heads] & & A_{k-1} \ar[rr] \ar[dr, two heads] & & \ldots \ar[rr] \ar[dr, two heads] & & A_{0}. \\
   & C_{k} \ar[ur, tail] & & C_{k-1} \ar[ur, tail] & \ldots & C_{1} \ar[ur, tail] & 
\end{tikzcd}\]
The sequence $A_k \to A_{k-1} \to \ldots \to A_0$ will be called
\begin{itemize}
\item acyclic, if $C_{i+1} \mono A_i \onto C_i$ is a short exact sequence in $\E$ for all $0 < i < k$.
\end{itemize}
An acyclic sequence in $\E$ as above is called
\begin{itemize}
\item left exact, if $A_k \to A_{k-1}$ is an admissible monomorphism (equivalently, $A_k \isoto C_k$),
\item right exact, if $A_1 \to A_0$ is an admissible epimorphism (i.e., $C_1 \isoto A_0$),
\item exact, if it is both left exact and right exact.
\end{itemize}
\end{defn}

\begin{lemma}\label{stringent}
Let $\E$ be a pointed category. The following conditions are equivalent.
\begin{enumerate}[label=$(\roman*)$]
\item\label{str1} There exists a proto-exact structure on $\E$, in which a morphism is admissible only if it admits a kernel or a cokernel.
\item\label{str2} The class of all kernel-cokernel pairs defines a proto-exact structure on $\E$, and a morphism in $\E$ is admissible if and only if it admits a kernel and a cokernel.
\item\label{str3} The pushout of a kernel by a cokernel exists in $\E$ and is a kernel again, and the pullback of a cokernel by a kernel exists in $\E$ and is a cokernel again.

Furthermore, if a morphism in $\E$ admits a kernel and a cokernel, then it is strict, that is, it factors as the composition of a cokernel and a kernel.
\item\label{str4} The pushout of a kernel by a cokernel exists in $\E$ and is a kernel again, and the pullback of a cokernel by a kernel exists in $\E$ and is a cokernel again.

Moreover, if a map $f$ in $\E$ admits a kernel and a cokernel, then the natural map \[\operatorname{coim}(f) \to \im(f)\] is an isomorphism.
\end{enumerate}
\end{lemma}
\begin{proof}
Assume \ref{str1}. The proto-exact structure on $\E$ necessarily contains all kernel-cokernel pairs, since kernels and cokernels are admissible morphisms. Moreover, admissible morphisms admit kernels and cokernels, implying \ref{str2}. The converse is tautological.

The equivalence \ref{str3} $\iff$ \ref{str4} is immediate from the fact that a factorization as in \ref{str3} is unique (up to unique isomorphism).

Finally, assuming \ref{str4}, we need to see that the class of all kernel-cokernel pairs defines a proto-exact structure on $\E$. Indeed, a composition of cokernels $B \onto B' \onto C$ admits a kernel
\[\begin{tikzcd}
A \ar[r, two heads] \ar[d, tail] \ar[dr, phantom, "{\pullback}" very near start] & A' \ar[r, two heads] \ar[d, tail] \ar[dr, phantom, "{\Box}"] & 0 \ar[d, tail] \\
B \ar[r, two heads] & B' \ar[r, two heads] & C
\end{tikzcd}\]
and therefore is a cokernel in $\E$. Then we can conclude that \ref{str4} $\iff$ \ref{str2} by definition.
\end{proof}

\begin{defn}
A pointed category $\E$ satisfying the equivalent conditions in Lemma \ref{stringent} will be called a stringent category.
\end{defn}

\begin{rem}\label{protoab}
In \cite{Toby}, Definition 1.2, a proto-abelian category is defined as a pointed category on which the classes of all monomorphisms and epimorphisms define a proto-exact structure. For us, it will prove convenient to change this terminology slightly by additionally requiring the existence of all kernels and cokernels (rather than introducing another different term). This does not exclude any of the main examples of interest (like Example \ref{strex} \ref{strex1} below).
\end{rem}

\begin{ex}\label{strex}
A pointed category which admits all kernels and cokernels is stringent if and only if it is proto-abelian. Similarly, a pre-abelian category is stringent if and only if it is abelian.
\begin{enumerate}[label=$(\arabic*)$]
\item\label{strex1} In particular, the category of (finite) $\FF_1$-vector spaces is stringent.

\item Consider the pointed category $\E$ on a non-zero object $V$ with $\End(V) \minus \{0\} = \{\epsilon\}$, such that $\epsilon^2 = 0$. It can easily be verified directly that $\epsilon$ admits neither a kernel nor a cokernel, and thus $\E$ is stringent.

If $F$ is a field, then the $F$-linear Cauchy completion of $\E$ is an additive stringent category, namely the category of finite free $F[x]/(x^2)$-modules.
\end{enumerate}
\end{ex}

\begin{rem}\label{wic}
An additive category $\E$ is stringent if and only if the class of all kernel-cokernel pairs defines an exact structure on $\E$ and a morphism in $\E$ is admissible if and only if it admits a kernel and a cokernel. Similarly, the other conditions in Lemma \ref{stringent} have evident additive analogues.

An additive stringent category $\E$ is in particular weakly idempotent complete, in the sense of \cite{ExCat}, Proposition 7.6. In fact, $\E$ satisfies the stronger statement of \cite{Schn}, Proposition 1.1.8, whose proof applies verbatim. Indeed, note that if a composition of the form $A' \onto A \xto{f} B$ is admissible, then its cokernel
\begin{equation}\label{cokercoker}
  \begin{tikzcd}
    A' \ar[r, two heads] \ar[d, two heads] \ar[dr, phantom, "{\pushout}" very near end] & A \ar[r, "{f}"] \ar[d, two heads] \ar[dr, phantom, "{\pushout}" very near end] & B \ar[d, two heads] \\
    0 \ar[r, two heads] & 0 \ar[r, two heads] & C
  \end{tikzcd}
\end{equation}
also defines a cokernel for $f$, which is therefore strict, and thus admits a kernel. In particular, the snake lemma for admissible morphisms holds in $\E$, as shown in \cite{ExCat}, Corollary 8.13.

In fact, the snake lemma holds in any stringent category $\E$. The neat argument presented in \textit{loc.cit.} does not quite apply here (as it ultimately relies on the additive structure of an exact category); however, the proof of \cite{Heller}, Proposition 4.3, does.

For this, we need to verify Heller's axioms for $\E$ (cf. \cite{ExCat}, Proposition B.1; with the obvious exception of additivity). Since cancellation follows directly from \eqref{cokercoker} by the coimage-image isomorphism, it only remains to prove the following result.
\end{rem}

\begin{prop}
Let $\E$ be a stringent category, and consider a diagram of the form
\[\begin{tikzcd}
  A' \ar[r, tail] \ar[d] & B' \ar[r, two heads] \ar[d, tail] & C' \ar[d, tail] \\
  A \ar[r, tail] \ar[d] & B \ar[r, two heads] \ar[d, two heads] & C \ar[d, two heads] \\
  A'' \ar[r, tail] & B'' \ar[r, two heads] & C''
\end{tikzcd}\]
in $\E$, where all rows as well as all columns but the first are exact. Then $A' \mono A \onto A''$ is a short exact sequence as well.
\end{prop}
\begin{proof}
First of all, by cancellation, $A' \mono A$. Then we have the following cartesian squares.
\[\begin{tikzcd}
  A' \ar[r, tail] \ar[d, two heads] \ar[dr, phantom, "{\Box}"] & B' \ar[d, two heads] & A' \ar[r, tail] \ar[d, tail] & B' \ar[d, tail] & A' \ar[r, tail] \ar[d, tail] & A \ar[d, tail] & A' \ar[r, tail] \ar[d, two heads] & A \ar[d] \\
  0 \ar[r, tail] \ar[d] \ar[dr, phantom, "{\pullback}" very near start] & C' \ar[d, tail] & A \ar[r, tail] \ar[d, two heads] \ar[dr, phantom, "{\Box}"] & B \ar[d, two heads] & B' \ar[r, tail] \ar[d, two heads] \ar[dr, phantom, "{\Box}"] & B \ar[d, two heads] & 0 \ar[r, tail] \ar[d] \ar[dr, phantom, "{\pullback}" very near start] & A'' \ar[d, tail] \\
  0 \ar[r, tail] & C & 0 \ar[r, tail] & C & 0 \ar[r, tail] & B'' & 0 \ar[r, tail] & B''
\end{tikzcd}\]
The first diagram implies that the outer rectangle of the second is a pullback, hence so is its upper square, which agrees with the upper square of the third, implying that the outer rectangle of the fourth is a pullback, and thus its upper square. Therefore,
\[\begin{tikzcd}[column sep=1.5em]
  \coker(A' \mono A) \ar[r, "{\sim}"] & \coker(\ker(A \to A'')) \ar[r, "{\sim}"] & \ker(\coker(A \to A'')) \ar[r, tail] & A''.
\end{tikzcd}\]
On the other hand, dually, we have the following cocartesian squares,
\[\begin{tikzcd}
  B' \ar[r, two heads] \ar[d, two heads] \ar[dr, phantom, "{\pushout}" very near end] & C' \ar[r, tail] \ar[d, two heads] \ar[dr, phantom, "{\Box}"] & C \ar[d, two heads] & B' \ar[r, tail] \ar[d, two heads] \ar[dr, phantom, "{\Box}"] & B \ar[r, two heads] \ar[d, two heads] & C \ar[d, two heads] \\
  0 \ar[r] & 0 \ar[r, tail] & C'' & 0 \ar[r, tail] & B'' \ar[r, two heads] & C''
\end{tikzcd}\]
which together imply that the right-hand square of the second diagram is again a pushout. But then the dual of Lemma \ref{monopb} below tells us that $A \onto A''$.
\end{proof}

\begin{lemma}\label{monopb}
Let $\E$ be a stringent category, and consider a pullback square of admissible morphisms in $\E$ of the following form.
\begin{equation}\label{monosquare}
  \begin{tikzcd}
    B_1 \ar[r, tail] \ar[d] \ar[dr, phantom, "{\pullback}" very near start] & A_1 \ar[d] \\
    B_0 \ar[r, tail] & A_0
  \end{tikzcd}
\end{equation}
The induced map $C_1 \mono C_0$ is an admissible monomorphism, where $C_i = \coker(B_i \mono A_i)$.
\end{lemma}
\begin{proof}
We have the following composition of cartesian squares.
\[\begin{tikzcd}
  B_1 \ar[r, tail] \ar[d] \ar[dr, phantom, "{\pullback}" very near start] & A_1 \ar[d] \\
  B_0 \ar[r, tail] \ar[d, two heads] \ar[dr, phantom, "{\Box}"] & A_0 \ar[d, two heads] \\
  0 \ar[r, tail] & C_0
\end{tikzcd}\]
Thus, the composition $A_1 \to A_0 \onto C_0$ is admissible with kernel $B_1$, and therefore admits a factorization $A_1 \onto C_1 \mono C_0$, which fits uniquely into the diagram
\[\begin{tikzcd}
  A_1 \ar[r, two heads] \ar[d, tail] & C_1 \ar[d, tail] \\
  A_0 \ar[r, two heads] & C_0
\end{tikzcd}\]
yielding the claim.
\end{proof}

%Let $B'_1 \mono B_1 \onto B'_0 \mono B_0$ and $A'_1 \mono A_1 \onto A'_0 \mono A_0$ be the corresponding factorizations together with their kernels. Since $B_1 \mono A_1$, there is an induced map $B'_1 \to A'_1$ and thus also $B'_0 \to A'_0$. But in fact, $B'_0 \mono A'_0$ is an admissible monomorphism, because $B_0 \mono A_0$ is, and $\E$ is weakly idempotent complete (\cite{ExCat}, Corollary 7.7).
%
%We claim that $C_1 \onto C'_0 \mono C_0$, where $C'_0 = \coker(B'_0 \mono A'_0)$.
%
%Let $D = \ker(C_1 \to C_0)$. Then we have the following diagram.
%\[\begin{tikzcd}
%  B_1 \ar[r, tail] \ar[d, two heads] \ar[dr, phantom, "{\Box}"] & B_1 \ar[r, tail] \ar[d, two heads] \ar[dr, phantom, "{\pullback}" very near start] & A_1 \ar[d, two heads] \\
%  0 \ar[r, tail] & D \ar[r, tail] \ar[d] \ar[dr, phantom, "{\pullback}" very near start] & C_1 \ar[d] \\
%   & 0 \ar[r, tail] & C_0
%\end{tikzcd}\]
%The lower square pulls back to the top and then to the left, by the above and by definition, respectively. But since pullbacks preserve admissible epimorphisms, we can infer that $D \isom \coker(B_1 \to B_1) = 0$, indeed.

\section{The higher Waldhausen construction}\label{sec5}

Let $\E$ be a proto-exact category, with its subcategory of admissible monomorphisms, resp. admissible epimorphisms, denoted by $\E^{\triangleleft}$, resp. $\E^{\triangleright}$.

Let $k,n \geq 0$. We write $\Fun([k],[n])$ for the category of functors between the standard ordinals $[k]$ and $[n]$, considered as small categories. Note that the objects of this category correspond bijectively to the set of $k$-simplices of the simplicial set $\Delta^n$.

\begin{defn}\label{waldhausen}
Let $k \geq 0$. For every $n \geq 0$, we define the category
\[S^{\bico k}_n(\E) \subseteq \Fun(\Fun([k],[n]), \E)\]
to be the full subcategory spanned by all diagrams $A$ satisfying the following conditions.
\begin{enumerate}[label=$(\alph*)$]
\item\label{degeneracy} For every $(k-1)$-simplex $\alpha$ in $\Delta^n$, we have
\[A_{s_{k-1}^* \alpha} = \ldots = A_{s_{0}^* \alpha} = 0.\]
\item For every $(k+1)$-simplex $\gamma$ in $\Delta^n$, the corresponding sequence
\[A_{d_{k+1}^* \gamma} \longto A_{d_{k}^* \gamma} \longto \dots \longto A_{d_1^* \gamma} \longto A_{d_{0}^* \gamma}\]
is acyclic.
\end{enumerate}
We define $S^{\lco k}_n(\E)$, resp. $S^{\rco k}_n(\E)$, as the full subcategory of $S^{\bico k}_n(\E)$ on all $A$ such that
\begin{enumerate}[label=$(b')$]
\item For every $(k+1)$-simplex $\gamma$ in $\Delta^n$, the sequence
\[A_{d_{k+1}^* \gamma} \longto A_{d_{k}^* \gamma} \longto \dots \longto A_{d_1^* \gamma} \longto A_{d_{0}^* \gamma}\]
is left exact, resp. right exact.
\end{enumerate}
Finally, we introduce $S^{\pair k}_n(\E) \subseteq S^{\bico k}_n(\E)$ as the full subcategory of diagrams $A$ which satisfy
\begin{enumerate}[label=$(b'')$]
\item For every $(k+1)$-simplex $\gamma$ in $\Delta^n$, the sequence
\[A_{d_{k+1}^* \gamma} \longto A_{d_{k}^* \gamma} \longto \dots \longto A_{d_1^* \gamma} \longto A_{d_{0}^* \gamma}\]
is exact.
\end{enumerate}
By functoriality in $[n]$, we obtain simplicial categories
\[S^{\pair k}(\E),\ S^{\lco k}(\E),\ S^{\rco k}(\E),\ S^{\bico k}(\E) \in \Cat_{\Delta}.\]
We call $S^{\pair k}(\E)$ the $k$-dimensional Waldhausen construction of $\E$.
\end{defn}

\begin{rem}
The $(k+1)$-skeleton of $S^{\pair k}(\E)$ has an immediate description. Namely,
\[S^{\pair k}_0(\E) = \ldots = S^{\pair k}_{k-1}(\E) = 0,\]
while $S^{\pair k}_k(\E) \isom \E$, and $S^{\pair k}_{k+1}(\E)$ is equivalent to the category of $k$-extensions in $\E$.
\end{rem}

\begin{ex}\label{waldhausenex}
Let $k \geq 0$, and let $\E$ be a proto-exact category.
\begin{enumerate}[label=$(\arabic*)$]
\item\label{S0} For $k=0$, the degeneracy condition \ref{degeneracy} is empty, and therefore,
\[S^{\pair 0}(\E) \isom N^{\E}(\E^{\times}) \isom \E\]
is the nerve of the maximal subgroupoid of $\E$, categorified by arbitrary morphisms in $\E$, which is equivalent to the constant object $\E$ itself. Similarly, $S^{\lco 0}(\E) \isom N^{\E}(\E^{\triangleleft})$, and dually, $S^{\rco 0}(\E) \isom N^{\E}(\E^{\triangleright})$. Rather more subtly, $S^{\bico 0}(\E) \isom N(\E)$ if and only if $\E$ is proto-abelian (in the sense of Remark \ref{protoab}), by \cite{Freyd}, Proposition 3.1.

\item\label{S1} For $k=1$, we recover a version of the original construction $S^{\pair 1}(\E) = S(\E)$ from \cite{Wa}, whose $n$-cells are given by the category formed by strictly upper triangular diagrams with bicartesian squares, as follows.
\begin{equation}\label{S1shape}
  \begin{tikzcd}
    0 \ar[r, tail] & A_{01} \ar[r, tail] \ar[d, two heads] \ar[dr, phantom, "{\Box}"] & A_{02} \ar[r, tail] \ar[d, two heads] \ar[dr, phantom, "{\Box}"] & \smash{\mbox{\raisebox{-1pt}{\ensuremath{\cdots}}}}\vphantom{A_{00}} \ar[r, tail] \ar[d, two heads] \ar[dr, phantom, "{\mbox{\raisebox{-10pt}{\ensuremath{\Box}}}}"] & A_{0n} \ar[d, two heads] \\
     & 0 \ar[r, tail] & A_{12} \ar[r, tail] \ar[d, two heads] \ar[dr, phantom, "{\Box}"] & \smash{\mbox{\raisebox{-2pt}{\ensuremath{\ddots}}}}\vphantom{A_{00}} \ar[r, tail] \ar[d, two heads] \ar[dr, phantom, "{\!\!\!\Box}"] & A_{1n} \ar[d, two heads] \\
     & & 0 \ar[r, tail] & \smash{\mbox{\raisebox{-2pt}{\ensuremath{\ddots}}}}\vphantom{A_{00}} \ar[r, tail] \ar[d, two heads] \ar[dr, phantom, "{\ \mbox{\raisebox{-10pt}{\ensuremath{\Box}}}}"] & \smash{\mbox{\raisebox{-3pt}{\ensuremath{\vdots}}}}\vphantom{A_{00}} \ar[d, two heads] \\
     & & & 0 \ar[r, tail] & A_{(n-1)n} \ar[d, two heads] \\
     & & & & 0
  \end{tikzcd}
\end{equation}
This is a refinement of Quillen's foundational construction $Q(\E) = N^{\E}(q(\E))$ in \cite{Quill}, where $q(\E)$ denotes the category of correspondences in $\E$ of the form $C \twoheadleftarrow B \mono A$. The forgetful functor from the edgewise subdivision $eS(\E) \to Q(\E)$ is an equivalence, that is, the whole diagram $A \in S_{2n+1}(\E)$ of shape \eqref{S1shape} can be recovered from
\begin{equation}
  \begin{tikzcd}
    & \ar[dl, two heads] \mathclap{\hspace{1em} A_{(n-1)(n+1)}} \phantom{A_{\ldots}} \ar[dr, tail] & & \ar[dl, two heads] \ldots \ar[dr, tail] & & \ar[dl, two heads] \mathclap{\hspace{1em} A_{0(2n)}} \phantom{A_{\ldots}} \ar[dr, tail] & \\
    \mathclap{\hspace{1em} A_{n(n+1)}} \phantom{A_{\ldots}} & & \mathclap{\hspace{1em} A_{(n-1)(n+2)}} \phantom{A_{\ldots}} & \ldots & \mathclap{\hspace{1em} A_{1(2n)}} \phantom{A_{\ldots}} & & \mathclap{\hspace{1em} A_{0(2n+1)}} \phantom{A_{\ldots}}
  \end{tikzcd}
\end{equation}
by taking successive pullbacks and pushouts in $\E$.

\item For $k=2$, the simplicial category $S^{\pair 2}(\E)$ was introduced by Hesselholt-Madsen \cite{HM}. An element $A \in S^{\pair 2}_4(\E)$ of its $4$-cells is a diagram of the following form.
\begin{equation}\label{4sim3d}
  \begin{tikzcd}[row sep=small, column sep=small]
    A_{012} \ar[rr, tail] & & A_{013} \ar[rr, tail] \ar[dd] \ar[ddrr, phantom, "{\pullback}" very near start] & & A_{014} \ar[dd] & & \\
     & \phantom{A_{000}} & & & & & \\
     & & A_{023} \ar[rr, tail] \ar[dr, two heads] & & A_{024} \ar[dd] \ar[dr, two heads] \ar[dddr, phantom, "{\pushout}" very near end] & & \\
     & & & A_{123} \ar[rr, tail, crossing over] & & A_{124} \ar[dd] & \\
     & & & & A_{034} \ar[dr, two heads] & & \\
     & & & & & A_{134} \ar[dr, two heads] & \\
     & & & & & & A_{234}
  \end{tikzcd}
\end{equation}
Note that the middle square is neither cartesian nor cocartesian. Rather, the diagram consists of bicartesian cubes (cf. Remark \ref{cubes}), as indicated in the following picture.
\begin{equation}\label{4simcubes}
  \begin{tikzcd}[row sep=small, column sep=small]
    A_{012} \ar[rr, tail] \ar[dd] \ar[dr, two heads] & & A_{013} \ar[rr, tail] \ar[dd] \ar[dr, two heads] & & A_{014} \ar[dd] \ar[dr, two heads] & & \\
     & {}_{\phantom{0}}0_{\phantom{0}} \ar[rr, tail, crossing over] & & {}_{\phantom{0}}0_{\phantom{0}} \ar[rr, tail, crossing over] & & {}_{\phantom{0}}0_{\phantom{0}} \ar[dd] & \\
    {}_{\phantom{0}}0_{\phantom{0}} \ar[rr, tail] \ar[dr, two heads] & & A_{023} \ar[rr, tail] \ar[dd] \ar[dr, two heads] & & A_{024} \ar[dd] \ar[dr, two heads] & & \\
     & {}_{\phantom{0}}0_{\phantom{0}} \ar[rr, tail, crossing over] \ar[from=uu, crossing over] & & A_{123} \ar[rr, tail, crossing over] \ar[from=uu, crossing over] & & A_{124} \ar[dd] \ar[dr, two heads] & \\
     & & {}_{\phantom{0}}0_{\phantom{0}} \ar[rr, tail] \ar[dr, two heads] & & A_{034} \ar[dr, two heads] & & {}_{\phantom{0}}0_{\phantom{0}} \ar[dd] \\
     & & & {}_{\phantom{0}}0_{\phantom{0}} \ar[rr, tail] \ar[from=uu, crossing over] \ar[dr, two heads] & & A_{134} \ar[dr, two heads] & \\
     & & & & {}_{\phantom{0}}0_{\phantom{0}} \ar[rr, tail] & & A_{234}
  \end{tikzcd}
\end{equation}
\end{enumerate}
\end{ex}

\begin{rem}\label{cubes}
In general, $S^{\pair k}(\E)$ is composed of $(k+1)$-dimensional bicartesian hypercubes. More precisely, let $A \in S^{\bico k}_n(\E)$. Then $A$ lies in $S^{\lco k}_n(\E)$ if and only if
\begin{equation}\label{cubeiso}
A_{\beta}\ = \lim\limits_{\substack{\beta < \beta' \\ \beta'-\beta \leq 1}}\! A_{\beta'}
\end{equation}
for every $k$-simplex $\beta$ in $\Delta^n$ with $\beta_{k-i} < n-i$ for $0 \leq i \leq k$. Then Lemma \ref{waldhausenop} below implies that the dual condition defines $S^{\rco k}_n(\E)$ inside $S^{\bico k}_n(\E)$, so that $A \in S^{\rco k}_n(\E)$ if and only if
\[\!\colim\limits_{\substack{\beta' < \beta \\ \beta-\beta' \leq 1}}\! A_{\beta'}\ =\ A_{\beta}\]
for every $k$-simplex $\beta$ in $\Delta^n$ with $i < \beta_i$ for all $0 \leq i \leq k$. Together, these yield the claim for
\[S^{\pair k}_n(\E) = S^{\lco k}_n(\E) \times_{S^{\bico k}_n(\E)} S^{\rco k}_n(\E).\]
In order to see the first statement, note that the sequence of admissible morphisms $A_{d_{\bullet}^*\gamma}$ corresponding to some $(k+1)$-simplex $\gamma$ in $\Delta^n$ defines a hypercube $\conv(A_{d_{\bullet}^*\gamma})$, formed by all $A_{\beta}$ with $d_{k+1}^*\gamma \leq \beta \leq d_0^*\gamma$. If the maps \eqref{cubeiso} are isomorphisms, the minimal subhypercubes of $\conv(A_{d_{\bullet}^*\gamma})$ are cartesian, and hence so is $\conv(A_{d_{\bullet}^*\gamma})$ as their composition. Therefore,
\[A_{d_{k+1}^*\gamma}\ = \!\!\lim\limits_{d_{k+1}^*\gamma < \beta \leq d_0^*\gamma}\!\! A_{\beta}\ =\ \ker(A_{d_{k}^*\gamma} \to A_{d_{k-1}^*\gamma}).\]
Conversely, let $\beta$ be a $k$-simplex in $\Delta^n$ with $\beta_{k-i} < n-i$ for all $0 \leq i \leq k$. Then we can infer inductively that the hypercube $Q_{\beta}$ on all $\beta' \geq \beta$ with $\beta'-\beta \leq 1$ is cartesian. First, assume that $|\beta| = \sum \beta_i$ is maximal. Thus, $\beta = (d_0^*)^{n-k}\Delta^n_n - 1 = d_{k+1}^*(d_0^*)^{n-k-1}\Delta^n_n$. But then $Q_{\beta}$ is exactly given by the hypercube $\conv(A_{d_{\bullet}^*(d_0^*)^{n-k-1}\Delta^n_n})$.

In general, consider $\gamma = \beta \disj \{n\}$. Then $d_{k+1}^*\gamma = \beta \leq \beta + 1 \leq d_0^*\gamma$, and hence $\conv(A_{d_{\bullet}^*\gamma})$ contains $Q_{\beta}$ entirely. But its complement is covered by hypercubes $Q_{\tilde{\beta}}$ with $|\tilde{\beta}| > |\beta|$, which are cartesian by induction. Therefore, so are all of their compositions, and thus so is $Q_{\beta}$.
\end{rem}

The following observation makes the inherent symmetry by duality precise.

\begin{lemma}\label{waldhausenop}
The duality on $\Delta$ induces equivalences of simplicial categories
\[\begin{aligned}
S^{\pair k}(\E) &\longto S^{\pair k}(\op\E), \\
S^{\lco k}(\E) &\longto S^{\rco k}(\op\E), \\
S^{\bico k}(\E) &\longto S^{\bico k}(\op\E).
\end{aligned}\]
\end{lemma}
\begin{proof}
This is immediate from the definitions.
\end{proof}

We can now state our main result.

\begin{thm}\label{waldhausen2k}
Let $k \geq 0$, and let $\E$ be a stringent category. The $k$-dimensional Waldhausen construction $S^{\pair k}(\E)$ is a fully $2k$-Segal category.
\end{thm}

For $k = 1$, this is one of the main results of \cite{HSS1}, namely Proposition 2.4.8. In light of work in progress by Bergner, Osorno, Ozornova, Rovelli, and Scheimbauer, another proof is provided by Example \ref{waldhausenex} \ref{S1}, where we have seen that $eS^{\pair 1}(\E)$ is lower $1$-Segal.

The case $k=0$ is settled by Example \ref{waldhausenex} \ref{S0}.

When $k = 2$, by Proposition \ref{implications}, we can deduce Theorem \ref{waldhausen2k} from Theorem \ref{S2} below. Its proof also serves to outline our strategy for the main theorem (up to the use of the double path space rather than both single path spaces).

An analogue of Theorem \ref{waldhausen2k} in the context of stable $\infty$-categories is a result of work in progress by Dyckerhoff and Jasso.

\begin{thm}\label{S2}
Let $\E$ be a proto-abelian category. Then $S^{\pair 2}(\E)$ is an upper $3$-Segal category.
\end{thm}
\begin{proof}
By the path space criterion, it suffices to show that $\lP \rP S^{\pair 2}(\E)$ is lower $1$-Segal. But Proposition \ref{pathspaces} below shows that for all $n \geq 2$, the forgetful functor
\[\lP \rP S^{\pair 2}_{n-2}(\E) \longto S^{\bico 0}_{n-2}(\E),\ A \longmapsto (A_{01n} \to A_{02n} \to \cdots \to A_{0(n-1)n}),\]
is an equivalence of categories, identifying the double path space $\lP \rP S^{\pair 2}(\E) \isoto N(\E)$ with the categorified nerve of $\E$, by Example \ref{waldhausenex} \ref{S0}.
\end{proof}

In \cite{HSS1}, the $1$-dimensional Waldhausen construction is shown to be fully $2$-Segal for any proto-exact category. However, when $k \geq 2$, our assumption is necessary, which we illustrate in Example \ref{not2k} below, at least in the additive case.

\begin{ex}\label{not2k}
Suppose that $\E$ is an exact category which is not stringent. Then $S^{\pair 2}(\E)$ is not $4$-Segal. Indeed, let us see that the lower $3$-Segal condition on the $4$-cells of the left path space $\lP S^{\pair 2}(\E)$ is violated.

By \cite{Freyd}, Proposition 3.1, there is a morphism $f\!: A \to B$ in $\E$ which is not strict. However, by \cite{ExCat}, Remark 8.2, it can be written as the composition of strict morphisms
\[A \xto{(1,f)} A \oplus B \xto{(0,1)} B.\]
Now suppose $f$ admits a kernel $C$ and consider the following possible element of $\lP S^{\pair 2}_4(\E)$.
\begin{equation}\label{PS2}
  \begin{tikzcd}
    0 \ar[r, tail] & 0 \ar[r, tail] \ar[d] \ar[dr, phantom, "{\pullback}" very near start] & C \ar[r, tail] \ar[d] \ar[dr, phantom, "{\pullback}" very near start] & C \ar[r, tail] \ar[d] \ar[dr, phantom, "{\pullback}" very near start] & A \ar[d, "{(1,f)}"] \\
     & 0 \ar[r, tail] & A \ar[r, tail] \ar[d] \ar[dr, phantom, "{\pullback}" very near start] & A \ar[r, tail] \ar[d] \ar[dr, phantom, "{\pullback}" very near start] & A \oplus B \ar[d, "{(0,1)}"] \\
     & & 0 \ar[r, tail] & 0 \ar[r, tail] \ar[d] \ar[dr, phantom, "{\pullback}" very near start] & B \ar[d] \\
     & & & 0 \ar[r, tail] & B \ar[d] \\
     & & & & 0
  \end{tikzcd}
\end{equation}
Then the triple of diagrams
\[\begin{tikzcd}[column sep=2em, row sep=1.5em]
  0 \ar[r, tail] \ar[d] \ar[dr, phantom, "{\pullback}" very near start] & C \ar[r, tail] \ar[d] \ar[dr, phantom, "{\pullback}" very near start] & C \ar[d] &  & 0 \ar[r, tail] \ar[d] \ar[dr, phantom, "{\pullback}" very near start] & C \ar[r, tail] \ar[d] \ar[dr, phantom, "{\pullback}" very near start] & A \ar[d, "{(1,f)}"] &  & A \ar[r, tail] \ar[d] \ar[dr, phantom, "{\pullback}" very near start] & A \ar[r, tail] \ar[d] \ar[dr, phantom, "{\pullback}" very near start] & A \oplus B \ar[d, "{(0,1)}"] \\
  0 \ar[r, tail] & A \ar[r, tail] \ar[d] \ar[dr, phantom, "{\pullback}" very near start] & A \ar[d] &  & 0 \ar[r, tail] & A \ar[r, tail] \ar[d] \ar[dr, phantom, "{\pullback}" very near start] & A \oplus B \ar[d] &  & 0 \ar[r, tail] & 0 \ar[r, tail] \ar[d] \ar[dr, phantom, "{\pullback}" very near start] & B \ar[d] \\
   & 0 \ar[r, tail] & 0 &  &  & 0 \ar[r, tail] & B &  &  & 0 \ar[r, tail] & B
\end{tikzcd}\]
defines an element in the right-hand side of the lower $3$-Segal map for $\lP S^{\pair 2}_4(\E)$. However, it does not lie in its essential image, because the sequence
\[\begin{tikzcd}[column sep=1.5em]
C \ar[r, tail] & A \ar[r, "{f}"] & B
\end{tikzcd}\]
indexed by $\{0,2,4\}$ is not left exact (the map $f$ not being strict), so $\eqref{PS2} \notin \lP S^{\pair 2}_4(\E)$.

Dually, since there exist non-strict morphisms admitting a cokernel in $\E$, the $4$-cells of the right path space $\rP S^{\pair 2}(\E)$ do not satisfy the upper $3$-Segal condition (by Lemma \ref{waldhausenop}).
\end{ex}

The above arguments also imply that the assumption in Theorem \ref{S2} is necessary (if $\E$ is additive). Namely, the lower $1$-Segal condition for $\lP \rP S^{\pair 2}(\E) \isom S^{\bico 0}(\E)$ requires admissible morphisms in $\E$ to be closed under composition, and thus $\E$ must be abelian already.

\begin{ex}
Let us illustrate the lowest $3$-Segal conditions for $S^{\pair 2}(\E)$, which is more conveniently done by depicting an element of its $4$-cells as the following projection of \eqref{4sim3d}.
\begin{equation}\label{4sim}
  \begin{tikzcd}[row sep=tiny]
     &   &   & A_{123} \ar[ddr, tail] &   &   &   \\
     &   &   & \phantom{A_{000}} &   &   &   \\
     &   & \color{red} A_{023} \ar[uur, two heads] \ar[dr, tail, color=red] &   & \color{red} A_{124} \ar[ddr] &   &   \\
     &   &   & \color{red} A_{024} \ar[ur, two heads, color=red] \ar[dr, color=red] &   &   &   \\
     & A_{013} \ar[uur] \ar[r, tail] & \color{red} A_{014} \ar[ur, color=red] \ar[rr] &   & \color{red} A_{034} \ar[r, two heads] \ar[ddrr, two heads, color=red] & A_{134} \ar[ddr, two heads] &   \\
     &   &   & \phantom{A_{000}} &   &   &   \\
    \color{red} A_{012} \ar[uur, tail] \ar[uurr, tail, color=red] &   &   &   &   &   & \color{red} A_{234}
  \end{tikzcd}
\end{equation}
The red part marks the image of \eqref{4sim} in the right-hand side of the upper $3$-Segal map
\[S^{\pair 2}_4(\E) \longto S^{\pair 2}_3(\E) \times_{S^{\pair 2}_2(\E)} S^{\pair 2}_3(\E).\]
The upper $3$-Segal condition says that the whole	diagram \eqref{4sim} can be uniquely recovered from the red subdiagram. Note that the complementary statement (for the lower $3$-Segal map) is false in general. In fact, it is equivalent to filling the frame of short exact sequences
\begin{equation}\label{frame}
  \begin{tikzcd}
    C_{4} \ar[r, tail] \ar[d, tail] & A_{023} \ar[r, two heads] \ar[d, tail, color=red] & A_{123} \ar[d, tail] \\
    C_{3} \ar[r, tail, color=red] \ar[d, two heads] & \color{red} A'_{024} \ar[r, two heads, color=red] \ar[d, two heads, color=red] & A_{124} \ar[d, two heads] \\
    C_{2} \ar[r, tail] & C_{1} \ar[r, two heads] & C_{0}
  \end{tikzcd}
\end{equation}
where $C_{i} = \coker(A_{d_3^*d_i^*\Delta^4_4} \mono A_{d_2^*d_i^*\Delta^4_4}) \isom \ker(A_{d_1^*d_i^*\Delta^4_4} \onto A_{d_0^*d_i^*\Delta^4_4})$. However, there is an obstruction to this, which is parametrized by the quotient groupoid
\[[\Ext^1(C_0,C_4)/\Hom(C_0,C_4)],\]
as calculated in \cite{Toby}, Lemma 2.30 and Proposition 2.38, assuming that $\E$ is abelian.
\end{ex}

\begin{rem}
Suppose $\E$ is additive. Then Theorem \ref{S2} does not generalize to the higher dimensional Waldhausen constructions, that is, $S^{\pair k}(\E)$ is not upper $(2k-1)$-Segal for $k \neq 2$. Indeed, the diagram
\begin{equation}
  \begin{tikzcd}
    0 \ar[r] & 0 \ar[r] \ar[d] & 0 \ar[r] \ar[d] & A \ar[r, "{=}"] \ar[d, "{(0,1)}"] & A \ar[d, "{=}"] \\
     & 0 \ar[r] & A \ar[r, "{(1,0)}"] \ar[d] & A \oplus A \ar[r, "{(1,1)}"'] \ar[d, "{(0,1)}"'] & A \ar[d] \\
     & & 0 \ar[r] & A \ar[r] \ar[d] & 0 \ar[d] \\
     & & & 0 \ar[r] & 0 \ar[d] \\
     & & & & 0
  \end{tikzcd}
\end{equation}
is an element of the right-hand side of the lower $3$-Segal map for $\lP\rP S^{\pair 3}_4(\E)$, but does not lie in its essential image.
\end{rem}

Our next observation will prove essential for our inductive arguments.

\begin{prop}[Hyperplane lemma]\label{hyperplanes}
Let $1 \leq k \leq l < m \leq n$, and let $\E$ be a stringent category. Then there is a natural functor
\[\eta^{\triangleleft}_{lm}\!: S^{\lco k}_n(\E) \longto S^{\lco{k-1}}_l(\E),\ A \longmapsto (\beta \mapsto \coker(A_{\beta \cup \{l\}} \mono A_{\beta \cup \{m\}})).\]
Dually, there is a corresponding natural functor
\[\eta^{\triangleright}_{lm}\!: S^{\rco k}_n(\E) \longto S^{\rco{k-1}}_{l}(\E),\ A \longmapsto (\beta \mapsto \ker(A_{\{n-m\} \cup \beta} \onto A_{\{n-l\} \cup \beta})).\]
Moreover, both of these restrict to functors on the higher Waldhausen construction,
\[\begin{tikzcd}
  S^{\pair k}_n(\E) \ar[r, shift left, "{\eta^{\triangleleft}_{lm}}"] \ar[r, shift right, "{\eta^{\triangleright}_{lm}}"'] & S^{\pair{k-1}}_l(\E).
\end{tikzcd}\]
\end{prop}
\begin{proof}
Let $\gamma$ be a $k$-simplex in $\Delta^{l}$, and $\gamma' = \gamma \disj \{m\}$. If $l \in \gamma$, then the sequence $\eta^{\triangleleft}_{lm}(A)_{d_{\bullet}^*\gamma}$ is given by
\[\begin{tikzcd}[column sep=1.5em]
\coker(A_{d_{k+1}^*\gamma'} \mono A_{d_{k}^*\gamma'}) \ar[r, tail] & A_{d_{k-1}^*\gamma'} \ar[r] & A_{d_{k-2}^*\gamma'} \ar[r] & \ldots \ar[r] & A_{d_{0}^*\gamma'}
\end{tikzcd}\]
which of course is indeed left exact. Furthermore, it is exact if and only if $A$ lies in $S^{\pair k}_n(\E)$, by definition.

Now assume that $l \notin \gamma$. Then, for each vertex $0 < i < k$, let us write
\[\begin{tikzcd}
  \mathllap{(A_{d_{i+1}^*\gamma \cup \{l\}} \onto )}\ B_{i+1} \ar[r, tail] \ar[d, dashed] & A_{d_i^*\gamma \cup \{l\}} \ar[r, two heads] \ar[d, dashed] & \ar[d, dashed] B_i\ \mathrlap{( \mono A_{d_{i-1}^*\gamma \cup \{l\}})} \\
  \mathllap{(A_{d_{i+1}^*\gamma \cup \{m\}} \onto )}\ C_{i+1} \ar[r, tail] & A_{d_i^*\gamma \cup \{m\}} \ar[r, two heads] & C_i\ \mathrlap{( \mono A_{d_{i-1}^*\gamma \cup \{m\}})}
\end{tikzcd}\]
for the corresponding short exact sequence. Taking cokernels yields a diagram as follows.
\[\begin{tikzcd}
  B_{i+1} \ar[r, tail] \ar[d, tail] & C_{i+1} \ar[r, two heads] \ar[d, tail] & D_{i+1} \ar[d] \\
  A_{d_i^*\gamma \cup \{l\}} \ar[r, tail] \ar[d, two heads] & A_{d_i^*\gamma \cup \{m\}} \ar[r, two heads] \ar[d, two heads] & \eta^{\triangleleft}_{lm}(A)_{d_i^*\gamma} \ar[d] \\
  B_i \ar[r, tail] & C_i \ar[r, two heads] & D_{i}
\end{tikzcd}\]
By the snake lemma, the right vertical sequence is short exact. Note that if $A \in S^{\pair k}_n(\E)$,
\[B_1 \isoto A_{d_{0}^*\gamma \cup \{l\}} \mbox{ and } C_1 \isoto A_{d_{0}^*\gamma \cup \{m\}}\]
which immediately implies also $D_1 \isoto \eta^{\triangleleft}_{lm}(A)_{d_0^*\gamma}$ by definition. It remains to prove that
\[\eta^{\triangleleft}_{lm}(A)_{d_k^*\gamma} \longto \eta^{\triangleleft}_{lm}(A)_{d_{k-1}^*\gamma}\]
is an admissible monomorphism. In order to see this, we may show that the diagram
\begin{equation}
\begin{tikzcd}
  A_{d_k^*\gamma \cup \{l\}} \ar[r, tail] \ar[d] & A_{d_k^*\gamma \cup \{m\}} \ar[d] \\
  A_{d_{k-1}^*\gamma \cup \{l\}} \ar[r, tail] & A_{d_{k-1}^*\gamma \cup \{m\}}
\end{tikzcd}
\end{equation}
is pullback, by Lemma \ref{monopb}. In fact, we claim that it is the composition of pullback diagrams
\[\begin{tikzcd}
  A_{d_k^*\gamma \cup \{l\}} \ar[r, tail] \ar[d] & A_{d_k^*\gamma \cup \{l+1\}} \ar[r, tail] \ar[d] & \ldots \ar[r, tail] & A_{d_k^*\gamma \cup \{m-1\}} \ar[r, tail] \ar[d] & A_{d_k^*\gamma \cup \{m\}} \ar[d] \\
  A_{d_{k-1}^*\gamma \cup \{l\}} \ar[r, tail] & A_{d_{k-1}^*\gamma \cup \{l+1\}} \ar[r, tail] & \ldots \ar[r, tail] & A_{d_{k-1}^*\gamma \cup \{m-1\}} \ar[r, tail] & A_{d_{k-1}^*\gamma \cup \{m\}}.
\end{tikzcd}\]
To prove this claim, for each $l \leq j < m$, we have the diagram
\[\begin{tikzcd}
 A_{d_k^*\gamma \cup \{j\}} \ar[r, tail] \ar[d] & A_{d_k^*\gamma \cup \{j+1\}} \ar[d] \\
 A_{d_{k-1}^*\gamma \cup \{j\}} \ar[r, tail] \ar[d] & A_{d_{k-1}^*\gamma \cup \{j+1\}} \ar[d] \\
 0 \ar[r, tail] & A_{d_{k-1}^*d_k^*\gamma \cup \{j,j+1\}}
\end{tikzcd}\]
in which the lower and outer rectangles are pullback, and therefore, so is the upper.

Finally, the functor $\eta^{\triangleright}_{lm}$ is given by the map $\eta^{\triangleleft}_{lm}$ for $\op\E$, via Lemma \ref{waldhausenop}.
\end{proof}

We are now prepared to compute the path spaces of the higher Waldhausen construction. This constitutes the main step in the proof of our main theorem, as well as a generalization of \cite{HSS1}, Lemma 2.4.9, which concerns the case $k=1$.

\begin{prop}\label{pathspaces}
Let $k \geq 1$, and assume that $\E$ is a stringent category. Then there are equivalences of simplicial categories
\[\begin{aligned}
\lP S^{\pair k}(\E) &\longto S^{\lco{k-1}}(\E),\ A \longmapsto A_{[0] \oplus -}, \\
\rP S^{\pair k}(\E) &\longto S^{\rco{k-1}}(\E),\ A \longmapsto A_{- \oplus [0]},
\end{aligned}\]
induced by the forgetful functors. For $k \geq 2$, there is an equivalence
\[\lP\rP S^{\pair k}(\E) \longto S^{\bico{k-2}}(\E),\ A \longmapsto A_{[0] \oplus - \oplus [0]}.\]
\end{prop}
\begin{proof}
We prove the second statement first. For a diagram $A \in S^{\rco{k-1}}_n(\E)$, we define its inverse image in $\rP S^{\pair k}_n(\E) = S^{\pair k}_{n+1}(\E)$ as a right Kan extension. Namely, we extend by zero appropriately, and then into the $k$th dimension, as follows.
\begin{equation}\label{Kanext}
  \begin{tikzcd}[column sep=7em]
    \Fun([k-1],[n]) \ar[r, "{A}"] \ar[d, "{\iota}"'] & \E \\
    \Fun([k-1],[n+1]) \ar[d, "{\operatorname{incl}}"'] & \\
    \operatorname{Cyl}(\iota|_{\Delta^n_{k-1}}) \ar[uur, "{A^!}", dashed] \ar[d, "{\lambda}"'] & \\
    \Fun([k],[n+1]) \ar[uuur, "{\hat A}", dashed, bend right=20]
  \end{tikzcd}
\end{equation}
Here, we have set $\iota = (d_{n+1})_*$, and the category $\operatorname{Cyl}(\iota|_{\Delta^n_{k-1}})$ is the cograph of its restriction to the skeleton. The functor $\lambda$ is defined by $(s_{k-1})_*$ on $\Delta^n_{k-1}$, and on $\Fun([k-1],[n+1])$, it maps
\[\alpha \mapsto \alpha \cup \{n+1\}.\]
Explicitly, $\hat A$ is given by the diagram
\[\beta \longmapsto \lim\limits_{\beta \leq \lambda(\alpha)} A^!_{\alpha} \ \isom \ \left\{\begin{array}{ll} A_{\beta \minus \{n+1\}} & \mbox{if } n+1 \in \beta, \\ \ker(A_{d_k^*\beta} \to A_{d_{k-1}^*\beta}) & \mbox{otherwise.} \end{array}\right.\]
Indeed, $d_{k}^*\beta$ is initial amongst those objects of the indexing category of the limit which come from $\Fun([k-1],[n+1])$. If $n+1 \in \beta$, then this is the only contribution. Otherwise, there are additionally the objects of the form
\[[\alpha] \in \Delta^n_{k-1} \mbox{ with } d_{k-1}^*\beta \leq \alpha.\] Therefore, in that case, the limit reduces to just the pullback
\[\hat A_{\beta} \isom \lim\left[\begin{tikzcd} & A^!_{d_{k}^*\beta} \ar[d] \\ A^!_{[d_{k-1}^*\beta]} \ar[r] & A^!_{d_{k-1}^*\beta} \end{tikzcd}\right] \isom \lim\left[\begin{tikzcd} & A_{d_k^*\beta} \ar[d] \\ 0 \ar[r] & A_{d_{k-1}^*\beta} \end{tikzcd}\right] = \ker(A_{d_k^*\beta} \to A_{d_{k-1}^*\beta}).\]
Now let $\gamma$ be a $(k+1)$-simplex in $\Delta^{n+1}$. We claim that the corresponding sequence $\hat{A}_{d_{\bullet}^*\gamma}$ is exact. If $n+1 \in \gamma$, then this is simply given by
\[\begin{tikzcd}[column sep=1.5em]
\ker(A_{d_k^*d_{k+1}^*\gamma} \to A_{d_{k-1}^*d_{k+1}^*\gamma}) \ar[r, tail] & A_{d_{k}^*\gamma \minus \{n+1\}} \ar[r] & \ldots \ar[r] & A_{d_{1}^*\gamma \minus \{n+1\}} \ar[r, two heads] & A_{d_{0}^*\gamma \minus \{n+1\}}
\end{tikzcd}\]
which is an exact sequence in $\E$ by definition.

\noindent Otherwise, the relevant sequence is given by
\[\begin{tikzcd}
 \ker(A_{d_k^*d_{k+1}^*\gamma} \to A_{d_{k-1}^*d_{k+1}^*\gamma}) \ar[d] &  & \ker(A_{d_{k}^*d_{k+1}^*\gamma} \to A_{d_{k-1}^*d_{k+1}^*\gamma}) \ar[d] \\
 \ker(A_{d_k^*d_{k}^*\gamma} \to A_{d_{k-1}^*d_{k}^*\gamma}) \ar[d] &  & \ker(A_{d_{k}^*d_{k+1}^*\gamma} \to A_{d_{k-1}^*d_{k}^*\gamma}) \ar[d] \\
 \ker(A_{d_k^*d_{k-1}^*\gamma} \to A_{d_{k-1}^*d_{k-1}^*\gamma}) \ar[d] & \phantom{\mbox{or equivalently,}} \ar[d, phantom, "{\mbox{or equivalently,}}"] & \ker(A_{d_{k-1}^*d_{k+1}^*\gamma} \to A_{d_{k-1}^*d_{k}^*\gamma}) \ar[d] \\
 \ker(A_{d_k^*d_{k-2}^*\gamma} \to A_{d_{k-1}^*d_{k-2}^*\gamma}) \ar[d] & \phantom{\mbox{or equivalently,}} & \ker(A_{d_{k-2}^*d_{k+1}^*\gamma} \to A_{d_{k-2}^*d_{k}^*\gamma}) \ar[d] \\
 \vdots \ar[d] &  & \vdots \ar[d] \\
 \ker(A_{d_k^*d_{0}^*\gamma} \to A_{d_{k-1}^*d_{0}^*\gamma}) &  & \ker(A_{d_{0}^*d_{k+1}^*\gamma} \to A_{d_{0}^*d_{k}^*\gamma}).
\end{tikzcd}\]
We prove exactness inductively. The first part of the sequence fits into a diagram of the form
\[\begin{tikzcd}
  \hat{A}_{d_{k+1}^*\gamma} \ar[r, tail] \ar[d] \ar[dr, phantom, "{\pullback}" very near start] & \hat{A}_{d^*_k\gamma_{\phantom{1}}} \ar[r, tail] \ar[d] \ar[dr, phantom, "{\pullback}" very near start] & A_{d_{k}^*d_{k+1}^*\gamma} \ar[d] \\
  0 \ar[r, tail] & \hat{A}_{d^*_{k-1}\gamma} \ar[r, tail] \ar[d] \ar[dr, phantom, "{\pullback}" very near start] & A_{d_{k-1}^*d_{k+1}^*\gamma} \ar[d] \\
   & 0 \ar[r, tail] & A_{d_{k-1}^*d_{k}^*\gamma}.
\end{tikzcd}\]
By definition, the bottom left square is pullback, so we can pull it back to the top and then to the left, since each outer rectangle is a pullback square by construction. Thus,
\[\begin{tikzcd}[column sep=1.5em]
\hat{A}_{d^*_{k+1}\gamma} \ar[r, tail] & \hat{A}_{d^*_{k}\gamma} \ar[r] & \hat{A}_{d^*_{k-1}\gamma}
\end{tikzcd}\]
is a left exact sequence. Now, for each $0 < i < k$, let us write
\begin{equation}\label{BiCi}
  \begin{tikzcd}
    \mathllap{(A_{d_{i+1}^*d_{k+1}^*\gamma} \onto )}\ B_{i+1} \ar[r, tail] \ar[d, dashed] & A_{d_i^*d_{k+1}^*\gamma} \ar[r, two heads] \ar[d, dashed] & \ar[d, dashed] B_i\ \mathrlap{( \mono A_{d_{i-1}^*d_{k+1}^*\gamma})} \\
    \mathllap{(A_{d_{i+1}^*d_{k}^*\gamma} \onto )}\ C_{i+1} \ar[r, tail] & A_{d_i^*d_{k}^*\gamma} \ar[r, two heads] & C_i\ \mathrlap{( \mono A_{d_{i-1}^*d_{k}^*\gamma})}
  \end{tikzcd}
\end{equation}
for the corresponding short exact sequences at the $i$th vertex of $d_{k+1}^*\gamma$ and $d_{k}^*\gamma$, respectively. First, we show that $B_k \to C_k$ is an admissible epimorphism. But we have
\[B_k = \coker(\hat{A}_{d^*_{k+1}\gamma} \mono A_{d_k^*d_{k+1}^*\gamma}), \mbox{ and } C_k = \coker(\hat{A}_{d^*_{k}\gamma} \mono A_{d_k^*d_{k}^*\gamma}).\]
Therefore, they fit into a diagram of the following form, which yields the claim.
\[\begin{tikzcd}
 \hat{A}_{d^*_{k+1}\gamma} \ar[r, tail] \ar[d, two heads] \ar[dr, phantom, "{\pushout}" very near end] & \hat{A}_{d^*_k \gamma} \ar[r, tail] \ar[d, two heads] \ar[dr, phantom, "{\pushout}" very near end] & A_{d_k^*d_{k+1}^*\gamma} \ar[d, two heads] \\
 0 \ar[r, tail] & B'_k \ar[r, tail] \ar[d, two heads] \ar[dr, phantom, "{\pushout}" very near end] & B_k \ar[d, two heads] \\
   & 0 \ar[r, tail] & C_k
\end{tikzcd}\]
In particular, $B'_k = \ker(B_k \onto C_k)$. Next, we show that $B_{k-1} \to C_{k-1}$ admits a kernel $B'_{k-1}$ in $\E$. In fact, consider the following diagram.
\[\begin{tikzcd}
 B_{k} \ar[r, tail] \ar[d, two heads] & A_{d_{k-1}^*d_{k+1}^*\gamma} \ar[r, two heads] \ar[d, two heads] & B_{k-1} \ar[d, two heads, dashed, "{(3)}"] \\
 C_{k} \ar[r, tail, dashed, "{(1)}"] \ar[d, "{=}"'] & E \ar[r, two heads, dashed, "{(1')}"] \ar[d, tail] & D \ar[d, tail, dashed, "{(2)}"] \\
 C_{k} \ar[r, tail] \ar[d, two heads] & A_{d_{k-1}^*d_{k}^*\gamma} \ar[r, two heads] \ar[d, two heads] & C_{k-1} \ar[d, two heads, dashed, "{(2')}"] \\
 0 \ar[r, tail] & C'_{k-1} \ar[r, "{=}"] & C'_{k-1}
\end{tikzcd}\]
We have $(1)$ by Remark \ref{wic}, and $(1')$ is its cokernel. The snake lemma yields $(2)$ and $(2')$, and $(3)$ is obtained dually to $(1)$. 

Now the snake lemma implies that the top row of the following diagram is short exact.
\[\begin{tikzcd}
 B'_{k} \ar[r, tail] \ar[d, tail] & \hat{A}_{d^*_{k-1}\gamma} \ar[r] \ar[d, tail] & B'_{k-1} \ar[d, tail] \\
 B_{k} \ar[r, tail] \ar[d, two heads] & A_{d_{k-1}^*d_{k+1}^*\gamma} \ar[r, two heads] \ar[d] & B_{k-1} \ar[d] \\
 C_{k} \ar[r, tail] & A_{d_{k-1}^*d_{k}^*\gamma} \ar[r, two heads] & C_{k-1}
\end{tikzcd}\]
In particular, this settles the case $k=2$. For $k \geq 3$, we can rewrite the diagram
\[\begin{tikzcd}
B_{2} \ar[r, tail] \ar[d] & A_{d_1^*d_{k+1}^*\gamma} \ar[r, two heads] \ar[d] & B_1\ \mathrlap{ = A_{d_{0}^*d_{k+1}^*\gamma}} \ar[d] \\
C_{2} \ar[r, tail] & A_{d_1^*d_{k}^*\gamma} \ar[r, two heads] & C_1\ \mathrlap{ = A_{d_{0}^*d_{k}^*\gamma}}
\end{tikzcd}\]
in terms of the hyperplane lemma (Proposition \ref{hyperplanes}), namely as the upper part of the short exact sequence of acyclic sequences
\begin{equation}\label{longsnake}
 \begin{tikzcd}
  \eta^{\triangleright}_{(n-\gamma_1)(n-\gamma_0)}(A)_{d_{k-1}^*\alpha} \ar[r, tail] \ar[d] & A_{\{\gamma_0\} \cup d_{k-1}^*\alpha} \ar[r, two heads] \ar[d] & A_{\{\gamma_1\} \cup d_{k-1}^*\alpha} \ar[d] \\
  \eta^{\triangleright}_{(n-\gamma_1)(n-\gamma_0)}(A)_{d_{k-2}^*\alpha} \ar[r, tail] \ar[d] & A_{\{\gamma_0\} \cup d_{k-2}^*\alpha} \ar[r, two heads] \ar[d] & A_{\{\gamma_1\} \cup d_{k-2}^*\alpha} \ar[d] \\
  \eta^{\triangleright}_{(n-\gamma_1)(n-\gamma_0)}(A)_{d_{k-3}^*\alpha} \ar[r, tail] \ar[d, end anchor={[yshift=2pt]}] & A_{\{\gamma_0\} \cup d_{k-3}^*\alpha} \ar[r, two heads] \ar[d, end anchor={[yshift=2pt]}] & A_{\{\gamma_1\} \cup d_{k-3}^*\alpha} \ar[d, end anchor={[yshift=2pt]}] \\
  \smash{\vdots} \vphantom{0} & \smash{\vdots} \vphantom{0} & \smash{\vdots} \vphantom{0}
 \end{tikzcd}
\end{equation}
where $\alpha = d_0^*d_1^*\gamma$. In particular, $B_2 \to C_2$ is an admissible morphism. Applying the snake lemma to the third morphism of short exact sequences in \eqref{longsnake} tells us that the map
\[C'_2 = \coker(B_2 \to C_2) \longto \coker(A_{d_1^*d_{k+1}^*\gamma} \to A_{d_1^*d_{k}^*\gamma})\]
is an admissible monomorphism, and therefore, by applying it to the first, that $\hat{A}_{d^*_1 \gamma} \onto \hat{A}_{d^*_0 \gamma}$.

Finally, we can iterate the argument, realizing \eqref{BiCi} as the upper part of the diagram
\[\begin{tikzcd}
  \eta^{(i)}(A)_{d_{k-i}^*\alpha^{(i)}} \ar[r, tail] \ar[d] & A_{\{\gamma_0,\ldots,\gamma_{i-1}\} \cup d_{k-i}^*\alpha^{(i)}} \ar[r, two heads] \ar[d] & \eta^{(i-1)}(A)_{d_{k-i+1}^*\alpha^{(i-1)}} \ar[d] \\
  \eta^{(i)}(A)_{d_{k-i-1}^*\alpha^{(i)}} \ar[r, tail] \ar[d] & A_{\{\gamma_0,\ldots,\gamma_{i-1}\} \cup d_{k-i-1}^*\alpha^{(i)}} \ar[r, two heads] \ar[d] & \eta^{(i-1)}(A)_{d_{k-i}^*\alpha^{(i-1)}} \ar[d] \\
  \eta^{(i)}(A)_{d_{k-i-2}^*\alpha^{(i)}} \ar[r, tail] \ar[d, end anchor={[yshift=2pt]}] & A_{\{\gamma_0,\ldots,\gamma_{i-1}\} \cup d_{k-i-2}^*\alpha^{(i)}} \ar[r, two heads] \ar[d, end anchor={[yshift=2pt]}] & \eta^{(i-1)}(A)_{d_{k-i-1}^*\alpha^{(i-1)}} \ar[d, end anchor={[yshift=2pt]}] \\
  \smash{\vdots} \vphantom{0} & \smash{\vdots} \vphantom{0} & \smash{\vdots} \vphantom{0}
\end{tikzcd}\]
where $\eta^{(i)} = \eta^{\triangleright}_{(n-\gamma_i)(n-\gamma_{i-1})} \circ \ldots \circ \eta^{\triangleright}_{(n-\gamma_1)(n-\gamma_0)}$ and $\alpha^{(i)} = d_0^* \ldots d_i^* \gamma$. Then the sequence
\[\begin{tikzcd}[column sep=1.5em]
B'_{i+1} \ar[r, tail] & \hat{A}_{d^*_{i} \gamma} \ar[r] & B'_{i}
\end{tikzcd}\]
is the beginning of the corresponding long exact snake, where $B'_{i} = \ker(B_i \to C_i)$, and is therefore a short exact sequence, as above.

Finally, the equivalence $\lP S^{\pair k}(\E) \isoto S^{\lco{k-1}}(\E)$ follows via Lemma \ref{waldhausenop} from the one we have proven above. Furthermore, if $k \geq 2$, we obtain $\lP\rP S^{\pair k}(\E) \isoto S^{\bico{k-2}}(\E)$ as an immediate consequence of the two. Namely, let $A \in S^{\bico{k-2}}_n(\E)$. Then the left Kan extension analogous to \eqref{Kanext} produces a diagram
\[\hat{A} \in S^{\rco{k-1}}_{n+1}(\E) = \lP S^{\rco{k-1}}_{n}(\E) \isom \lP\rP S^{\pair k}_n(\E),\]
as all arguments above apply verbatim to show that $\hat{A}$ consists of right exact sequences.
\end{proof}

\begin{rem}
Proposition \ref{pathspaces} can be seen as a higher analogue of the third isomorphism theorem, in that the equivalence of categories $S^{\lco{k-1}}_{k+1}(\E) \isoto \lP S^{\pair{k}}_{k+1}(\E) = S^{\pair{k}}_{k+2}(\E)$ boils down to the following statement. Given a configuration of left exact sequences of the form
\[\begin{tikzcd}[row sep=0em, column sep=0em, cramped]
  \phantom{0} \mathclap{\smash{A^{k+1}_k}} \phantom{0} \ar[rrrrr, tail, start anchor={[xshift=2pt]}, end anchor={[xshift=-1ex]}] \ar[ddrrrrrrrr, tail, start anchor={[xshift=4pt, yshift=-1pt]}, end anchor={[xshift=-4pt]}] & \phantom{00} & \phantom{00} & \phantom{00} & \phantom{00} & \phantom{0} \mathclap{\smash{A^{k+1}_{k-1}}} \phantom{0} \ar[rrrrr, start anchor={[xshift=1ex]}, end anchor={[xshift=1ex]}] \ar[ddrrr, tail, start anchor={[xshift=3pt, yshift=-2pt]}] & \phantom{00} & \phantom{00} & \phantom{00} & \phantom{00} & \phantom{00} \ar[rrrrr, end anchor={[xshift=-1ex]}] \ar[dddrr, dash, dotted, start anchor={[xshift=-2pt, yshift=3pt]}, end anchor={[xshift=2pt, yshift=-3pt]}] & \phantom{00} & \phantom{00} & \phantom{00} & \phantom{00} & \phantom{0} \mathclap{\smash{A^{k+1}_{1}}} \phantom{0} \ar[rrrrr, end anchor={[xshift=-1ex]}] \ar[ddddr, tail] & \phantom{00} & \phantom{00} & \phantom{00} & \phantom{00} & \phantom{0} \mathclap{\smash{A^{k+1}_0}} \phantom{0} \ar[ddddd, tail] \\
    & \phantom{00} & \phantom{00} & \phantom{00} & \phantom{00} & \phantom{00} & \phantom{00} & \phantom{00} & \phantom{00} & \phantom{00} & \phantom{00} & \phantom{00} & \phantom{00} & \phantom{00} & \phantom{00} & \phantom{00} & \phantom{00} & \phantom{00} & \phantom{00} & \phantom{00} & \phantom{00} \\
    & \phantom{00} & \phantom{00} & \phantom{00} & \phantom{00} & \phantom{00} & \phantom{00} & \phantom{00} & \phantom{0} \mathclap{\smash{A^{k}_{k-1}}} \phantom{0} \ar[drrrr, start anchor={[xshift=4pt, yshift=-1pt]}] \ar[ddddrrrrrr, start anchor={[xshift=3pt, yshift=-2pt]}] & \phantom{00} & \phantom{00} & \phantom{00} & \phantom{00} & \phantom{00} & \phantom{00} & \phantom{00} & \phantom{00} & \phantom{00} & \phantom{00} & \phantom{00} & \phantom{00} \\
    & \phantom{00} & \phantom{00} & \phantom{00} & \phantom{00} & \phantom{00} & \phantom{00} & \phantom{00} & \phantom{00} & \phantom{00} & \phantom{00} & \phantom{00} & \phantom{00} \ar[dddrr, dash, dotted, start anchor={[xshift=-2pt, yshift=3pt]}, end anchor={[xshift=2pt, yshift=-3pt]}] \ar[drrrr] & \phantom{00} & \phantom{00} & \phantom{00} & \phantom{00} & \phantom{00} & \phantom{00} & \phantom{00} & \phantom{00} \\
    & \phantom{00} & \phantom{00} & \phantom{00} & \phantom{00} & \phantom{00} & \phantom{00} & \phantom{00} & \phantom{00} & \phantom{00} & \phantom{00} & \phantom{00} & \phantom{00} & \phantom{00} & \phantom{00} & \phantom{00} & \phantom{0} \mathclap{\smash{A^{k}_{1}}} \phantom{0} \ar[drrrr] \ar[ddddr, end anchor={[xshift=-.5pt, yshift=2pt]}] & \phantom{00} & \phantom{00} & \phantom{00} & \phantom{00} \\
    & \phantom{00} & \phantom{00} & \phantom{00} & \phantom{00} & \phantom{00} & \phantom{00} & \phantom{00} & \phantom{00} & \phantom{00} & \phantom{00} & \phantom{00} & \phantom{00} & \phantom{00} & \phantom{00} & \phantom{00} & \phantom{00} & \phantom{00} & \phantom{00} & \phantom{00} & \phantom{0} \mathclap{\smash{A^{k}_{0}}} \phantom{0} \ar[ddddd] \\
    & \phantom{00} & \phantom{00} & \phantom{00} & \phantom{00} & \phantom{00} & \phantom{00} & \phantom{00} & \phantom{00} & \phantom{00} & \phantom{00} & \phantom{00} & \phantom{00} & \phantom{00} & \phantom{00} \ar[ddddddrrrr, dash, dotted, start anchor={[xshift=-2pt, yshift=3pt]}, end anchor={[xshift=2pt, yshift=-3pt]}] \ar[ddrrr] & \phantom{00} & \phantom{00} & \phantom{00} & \phantom{00} & \phantom{00} & \phantom{00} \\
    & \phantom{00} & \phantom{00} & \phantom{00} & \phantom{00} & \phantom{00} & \phantom{00} & \phantom{00} & \phantom{00} & \phantom{00} & \phantom{00} & \phantom{00} & \phantom{00} & \phantom{00} & \phantom{00} & \phantom{00} & \phantom{00} & \phantom{00} & \phantom{00} & \phantom{00} & \phantom{00} \\
    & \phantom{00} & \phantom{00} & \phantom{00} & \phantom{00} & \phantom{00} & \phantom{00} & \phantom{00} & \phantom{00} & \phantom{00} & \phantom{00} & \phantom{00} & \phantom{00} & \phantom{00} & \phantom{00} & \phantom{00} & \phantom{00} & \phantom{0} \mathclap{\smash{A^{k-1}_{1}}} \phantom{0} \ar[ddrrr] \ar[ddddr] & \phantom{00} & \phantom{00} & \phantom{00} \\
    & \phantom{00} & \phantom{00} & \phantom{00} & \phantom{00} & \phantom{00} & \phantom{00} & \phantom{00} & \phantom{00} & \phantom{00} & \phantom{00} & \phantom{00} & \phantom{00} & \phantom{00} & \phantom{00} & \phantom{00} & \phantom{00} & \phantom{00} & \phantom{00} & \phantom{00} & \phantom{00} \\
    & \phantom{00} & \phantom{00} & \phantom{00} & \phantom{00} & \phantom{00} & \phantom{00} & \phantom{00} & \phantom{00} & \phantom{00} & \phantom{00} & \phantom{00} & \phantom{00} & \phantom{00} & \phantom{00} & \phantom{00} & \phantom{00} & \phantom{00} & \phantom{00} & \phantom{00} & \phantom{0} \mathclap{\smash{A^{k-1}_0}} \phantom{0} \ar[ddddd] \\
    & \phantom{00} & \phantom{00} & \phantom{00} & \phantom{00} & \phantom{00} & \phantom{00} & \phantom{00} & \phantom{00} & \phantom{00} & \phantom{00} & \phantom{00} & \phantom{00} & \phantom{00} & \phantom{00} & \phantom{00} & \phantom{00} & \phantom{00} & \phantom{00} & \phantom{00} & \phantom{00} \\
    & \phantom{00} & \phantom{00} & \phantom{00} & \phantom{00} & \phantom{00} & \phantom{00} & \phantom{00} & \phantom{00} & \phantom{00} & \phantom{00} & \phantom{00} & \phantom{00} & \phantom{00} & \phantom{00} & \phantom{00} & \phantom{00} & \phantom{00} & \phantom{00} \ar[ddddddddrr, end anchor={[xshift=-1pt, yshift=4pt]}] \ar[dddrr, dash, dotted, start anchor={[xshift=-2pt, yshift=3pt]}, end anchor={[xshift=2pt, yshift=-3pt]}] & \phantom{00} & \phantom{00} \\
    & \phantom{00} & \phantom{00} & \phantom{00} & \phantom{00} & \phantom{00} & \phantom{00} & \phantom{00} & \phantom{00} & \phantom{00} & \phantom{00} & \phantom{00} & \phantom{00} & \phantom{00} & \phantom{00} & \phantom{00} & \phantom{00} & \phantom{00} & \phantom{00} & \phantom{00} & \phantom{00} \\
    & \phantom{00} & \phantom{00} & \phantom{00} & \phantom{00} & \phantom{00} & \phantom{00} & \phantom{00} & \phantom{00} & \phantom{00} & \phantom{00} & \phantom{00} & \phantom{00} & \phantom{00} & \phantom{00} & \phantom{00} & \phantom{00} & \phantom{00} & \phantom{00} & \phantom{00} & \phantom{00} \\
    & \phantom{00} & \phantom{00} & \phantom{00} & \phantom{00} & \phantom{00} & \phantom{00} & \phantom{00} & \phantom{00} & \phantom{00} & \phantom{00} & \phantom{00} & \phantom{00} & \phantom{00} & \phantom{00} & \phantom{00} & \phantom{00} & \phantom{00} & \phantom{00} & \phantom{00} & \phantom{00} \ar[ddddd] \\
    & \phantom{00} & \phantom{00} & \phantom{00} & \phantom{00} & \phantom{00} & \phantom{00} & \phantom{00} & \phantom{00} & \phantom{00} & \phantom{00} & \phantom{00} & \phantom{00} & \phantom{00} & \phantom{00} & \phantom{00} & \phantom{00} & \phantom{00} & \phantom{00} & \phantom{00} & \phantom{00} \\
    & \phantom{00} & \phantom{00} & \phantom{00} & \phantom{00} & \phantom{00} & \phantom{00} & \phantom{00} & \phantom{00} & \phantom{00} & \phantom{00} & \phantom{00} & \phantom{00} & \phantom{00} & \phantom{00} & \phantom{00} & \phantom{00} & \phantom{00} & \phantom{00} & \phantom{00} & \phantom{00} \\
    & \phantom{00} & \phantom{00} & \phantom{00} & \phantom{00} & \phantom{00} & \phantom{00} & \phantom{00} & \phantom{00} & \phantom{00} & \phantom{00} & \phantom{00} & \phantom{00} & \phantom{00} & \phantom{00} & \phantom{00} & \phantom{00} & \phantom{00} & \phantom{00} & \phantom{00} & \phantom{00} \\
    & \phantom{00} & \phantom{00} & \phantom{00} & \phantom{00} & \phantom{00} & \phantom{00} & \phantom{00} & \phantom{00} & \phantom{00} & \phantom{00} & \phantom{00} & \phantom{00} & \phantom{00} & \phantom{00} & \phantom{00} & \phantom{00} & \phantom{00} & \phantom{00} & \phantom{00} & \phantom{00} \\
    & \phantom{00} & \phantom{00} & \phantom{00} & \phantom{00} & \phantom{00} & \phantom{00} & \phantom{00} & \phantom{00} & \phantom{00} & \phantom{00} & \phantom{00} & \phantom{00} & \phantom{00} & \phantom{00} & \phantom{00} & \phantom{00} & \phantom{00} & \phantom{00} & \phantom{00} & \phantom{0} \mathclap{\smash{A^{0}_0}} \phantom{0}
\end{tikzcd}\]
where $A_i^j = A_{d_i^*d_j^*\Delta^{k+1}_{k+1}}$ in the previous notation, the induced maps between the cokernels
\[\coker(A^{k+1}_1 \to A^{k+1}_0) \longto \coker(A^{k}_1 \to A^{k}_0) \longto \ldots \longto \coker(A^{0}_1 \to A^{0}_0)\]
constitute an exact sequence in $\E$.
\end{rem}

We are now prepared to prove our main result.

\begin{proof}[Proof of Theorem \ref{waldhausen2k}]
Our goal is to show that $S^{\lco{k-1}}(\E)$ is a lower $(2k-1)$-Segal category. Then, by Proposition \ref{pathspaces} as well as Lemma \ref{waldhausenop}, the path space criterion completes the proof.

Let $n \geq 2k$. We will prove by induction that the lower $(2k-1)$-Segal map
\begin{equation}\label{Segalmap}
S^{\lco{k-1}}_n(\E) \longto \lim_{I\in\L([n],2k-1)} S^{\lco{k-1}}_{I}(\E)
\end{equation}
is an equivalence. Throughout, for $0 \leq i \leq n$, let $\delta_i$ refer to the $i$th face map of $\Delta^{n}_n$, even when applied to any subsimplex of it. That is, $\delta_i^*$ removes the vertex $i$, and $\hat{\delta}_{i}^*$ adjoins it.

First, consider the case $n = 2k$. By Lemma \ref{alreadyeven}, the only $k$-simplex in $\Delta^{2k}$ not contained in an even subset of $[2k]$ of cardinality $2k$ already is
\[\epsilon = \{0,2,\ldots,2k\} = \delta_{2k-1}^* \delta_{2k-3}^* \cdots \delta_1^* \Delta^{2k}_{2k}.\]
But if $A$ lies in the right-hand side of \eqref{Segalmap}, then we can form the unique compositions
\[\begin{tikzcd}[column sep=0]
  A_{\delta_{2k}^*\epsilon} \ar[rr, dashed] \ar[dr, tail] & & A_{\delta_{2k-2}^*\epsilon} \ar[rr, dashed] \ar[dr] & & \ldots \ar[rrrr, dashed] \ar[drr] & & & & A_{\delta_0^*\epsilon} \\
   & A_{\delta_{2k}^*\hat{\delta}_{2k-1}^*\delta_{2k-2}^*\epsilon} \ar[ur, tail] & & A_{\delta_{2k-2}^*\hat{\delta}_{2k-3}^*\delta_{2k-4}^*\epsilon} \ar[ur] & \ldots & \phantom{A} & A_{\delta_{2}^*\hat{\delta}_{1}^*\delta_{0}^*\epsilon} \ar[urr] & \phantom{A} &
\end{tikzcd}\]
completing $A$ to an element of $S^{\lco{k-1}}_{2k}(\E)$. It remains to be shown that the resulting sequence
\[A_{\delta_{2k}^*\epsilon} \longto A_{\delta_{2k-2}^*\epsilon} \longto \ldots \longto A_{\delta_0^*\epsilon}\]
is left exact. We proceed inductively. The case $k=2$ is settled by Theorem \ref{S2}. In general, since $A_{\delta_{2k}^*\epsilon} \mono A_{\delta_{2k-2}^*\epsilon}$ is an admissible monomorphism (as a composition of such), it suffices to prove that
\begin{equation}\label{evensequ}
\coker(A_{\delta_{2k}^*\epsilon} \mono A_{\delta_{2k-2}^*\epsilon}) \longto A_{\delta_{2k-4}^*\epsilon} \longto \ldots \longto A_{\delta_0^*\epsilon}
\end{equation}
is a left exact sequence. For this, we use the hyperplane lemma. Namely, the functor
\[\eta^{\triangleleft}_{(2k-2)2k}\!: S^{\lco{k-1}}_{2k}(\E) \longto S^{\lco{k-2}}_{2k-2}(\E)\]
constructed in Proposition \ref{hyperplanes} is compatible with the corresponding lower Segal maps on both sides, in that it induces a commutative diagram of the following form.
\[\begin{tikzcd}
 S^{\lco{k-1}}_{2k}(\E) \ar[r, end anchor={[xshift=1ex]}] \ar[d, "{\eta^{\triangleleft}_{(2k-2)2k}}"] & \lim\limits_{I\in\L([2k],2k-1)} \hspace{-1ex} S^{\lco{k-1}}_{I}(\E) \ar[d, dashed, shift left=3ex, start anchor={[yshift=2ex]}, "{\eta^{\triangleleft}_{(2k-2)2k}}"] \\
 S^{\lco{k-2}}_{2k-2}(\E) \ar[r, end anchor={[xshift=3ex]}] & \lim\limits_{J\in\L([2k-2],2k-3)} \hspace{-1ex}\! S^{\lco{k-2}}_{J}(\E)
\end{tikzcd}\]
Indeed, this is because we have $d_0^*d_0^*\epsilon = \delta_{2k-3}^* \delta_{2k-5}^* \cdots \delta_1^* \Delta^{2k-2}_{2k-2}$. But then, by induction, the lower horizontal map is an equivalence, which by the above means precisely that \eqref{evensequ} is a left exact sequence.

In order to prove the Segal conditions for the higher cells $S^{\lco{k-1}}_{n}(\E)$, we once again employ induction, now on the dimension $n$. If $A$ lies in the right-hand side of the lower $(2k-1)$-Segal map for $S^{\lco{k-1}}_{n}(\E)$, we first need to see that taking compositions completes $A$ to a well-defined diagram of shape $\Fun([k-1], [n])$ in $\E$.

By Lemma \ref{alreadyeven}, we need only consider sequences indexed by simplices $\gamma$ with all vertices separated by gaps. For gaps $i \in [n]$ of size $1$, there is (as before) a unique composition,
\[\begin{tikzcd}[column sep=0]
  A_{\delta_{i+1}^*\gamma} \ar[rr, dashed] \ar[dr] & & A_{\delta_{i-1}^*\gamma}. \\
   & A_{\delta_{i+1}^*\hat{\delta}_{i}^*\delta_{i-1}^*\gamma} \ar[ur] & 
\end{tikzcd}\]
For a gap of $\gamma$ of size $l+1$, say $\{i,\ldots,i+l\} \subseteq [n]$, each $0 \leq j \leq l$ defines the composition
\[\begin{tikzcd}[column sep=0]
  A_{\delta_{i+l+1}^*\gamma} \ar[rr, dashed] \ar[dr] & & A_{\delta_{i-1}^*\gamma}. \\
   & A_{\delta_{i+l+1}^*\hat{\delta}_{i+j}^*\delta_{i-1}^*\gamma} \ar[ur] & 
\end{tikzcd}\]
By induction, all possible compositions can be reduced to one of these. On the other hand, they all agree, since for all $0 \leq j < j' \leq l$, the following diagram commutes.
\[\begin{tikzcd}[column sep=0]
  A_{\delta_{i+l+1}^*\gamma} \ar[r] \ar[d] & A_{\delta_{i+l+1}^*\hat{\delta}_{i+j'}^*\delta_{i-1}^*\gamma} \ar[d] \\
  A_{\delta_{i+l+1}^*\hat{\delta}_{i+j}^*\delta_{i-1}^*\gamma} \ar[r] \ar[ur] & A_{\delta_{i-1}^*\gamma}
\end{tikzcd}\]
Finally, we apply induction to obtain the remaining exactness conditions for the completed diagram of $A$. Namely, the sequence indexed by $\gamma$ is left exact, since $\del_i(A)$ lies in the right-hand side of the lower $(2k-1)$-Segal map of $S^{\lco{k-1}}_{n-1}(\E)$, for any gap $i$ of $\gamma$.
\end{proof}

\begin{lemma}\label{alreadyeven}
Let $n \geq 2k$. Let $\gamma$ be a $k$-subsimplex of $\Delta^n$ with a pair of adjacent simplices. Then $\gamma$ is contained in an even subset $I \subseteq [n]$ of cardinality $\# I = 2k$.
\end{lemma}
\begin{proof}
For $n = 2k$, this is clear. For the induction step, we can assume that $n \in \gamma$, otherwise the statement follows tautologically from the induction hypothesis. Let $0 < m < n$ be the maximal gap of $\gamma$. By induction, $(\gamma \cup \{m\}) \minus \{n\}$ is contained in an even subset $I' \subseteq [n-1]$ with $\# I' = 2k$. But then $\gamma$ is contained in $I = (I' \minus \{m\}) \cup \{n\}$, which is even in $[n]$.
\end{proof}

\bigskip

\bibliography{higherS}
\bibliographystyle{plain}

\bigskip

\end{document}